\documentclass[a4paper,10pt]{amsart}
\usepackage{amsmath,amssymb,amsthm,amscd}
\usepackage[all]{xy}
\renewcommand{\iff}{if and only if }
\newcommand{\st}{such that }
\newcommand{\modR}{\hbox{\rm mod-}R}
\newcommand{\ModR}{\hbox{\rm Mod-}R}

\newcommand{\stModR}{\hbox{\rm \underline{Mod}-}R}

\newcommand{\Add}{\mathrm{Add}}
\newcommand{\Prod}{\mathrm{Prod}}

\newcommand{\Pot}{\mathfrak{P}}

\newcommand{\Plambda}{\mathfrak P^{<\lambda}}
\newcommand{\Ab}{\mathrm{Ab}}
\newcommand{\T}{\mathcal{T}}
\newcommand{\Tcomp}{{\mathcal{T}_0}}
\newcommand{\Z}{\mathbb{Z}}
\newcommand{\Q}{\mathbb{Q}}
\DeclareMathOperator{\Hom}{Hom}

\DeclareMathOperator{\Ext}{Ext}

\DeclareMathOperator{\Ker}{Ker}
\DeclareMathOperator{\Img}{Im}
\DeclareMathOperator{\Coker}{Coker}

\DeclareMathOperator{\cf}{cf}
\DeclareMathOperator{\rk}{rk}

\DeclareMathOperator{\hocolim}{\underrightarrow{\textrm{hocolim}}}

\theoremstyle{plain}
\newtheorem{thm}{Theorem}[section]
\newtheorem{prop}[thm]{Proposition}
\newtheorem{lem}[thm]{Lemma}
\newtheorem{cor}[thm]{Corollary}
\theoremstyle{definition}
\newtheorem{defn}[thm]{Definition}

\theoremstyle{remark}
\newtheorem*{rem}{Remark}

\begin{document}
\title{The countable Telescope Conjecture for module categories}

\author{\textsc{Jan \v Saroch}}
\address{Charles University, Faculty of Mathematics and Physics, Department of Algebra \\ 
Sokolovsk\'{a} 83, 186 75 Prague 8, Czech Republic}
\email{saroch@karlin.mff.cuni.cz}

\author{\textsc{Jan~\v S\v tov\'\i\v cek}}
\address{Institutt for matematiske fag, Norges teknisk-naturvitenskapelige universitet \\ 
N-7491 Trondheim, Norway}
\email{stovicek@math.ntnu.no}


\thanks{Both authors supported by the grants GAUK 301-10/252216, GA\v CR 201/05/H005, and the research project MSM 0021620839.}
\thanks{Second author supported also by the Research Council of Norway through Storforsk-project Homological and geometric methods in algebra.}

\date{\today}

% = Abstract ========================================================
\begin{abstract} By the Telescope Conjecture for Module Categories, we mean the following claim: ``Let $R$ be any ring and $(\mathcal A, \mathcal B)$ be a hereditary cotorsion pair in $\ModR$ with $\mathcal A$ and $\mathcal B$ closed under direct limits. Then $(\mathcal A, \mathcal B)$ is of finite type.''

We prove a modification of this conjecture with the word `finite' replaced by `countable'. We show that a hereditary cotorsion pair $(\mathcal A,\mathcal B)$ of modules over an arbitrary ring $R$ is generated by a set of strongly countably presented modules provided that $\mathcal B$ is closed under unions of well-ordered chains. We also characterize the modules in $\mathcal B$ and the countably presented modules in $\mathcal A$ in terms of morphisms between finitely presented modules, and show that $(\mathcal A, \mathcal B)$ is cogenerated by a single pure-injective module provided that $\mathcal A$ is closed under direct limits. Then we move our attention to strong analogies between cotorsion pairs in module categories and localizing pairs in compactly generated triangulated categories.
\end{abstract}

\maketitle
\vspace{4ex}

% = Introduction ====================================================

Motivated by the paper \cite{KS} of Krause and Solberg, the first author with Lidia Angeleri H\" ugel and Jan Trlifaj started in \cite{AST} an investigation of the \emph{Telescope Conjecture for Module Categories} (TCMC) stated as follows (see Section~\ref{sec:prelim} for unexplained terminology):

\medskip

{\bf Telescope Conjecture for Module Categories.} Let $R$ be a ring and $(\mathcal A, \mathcal B)$ be a hereditary cotorsion pair in $\ModR$ with $\mathcal A$ and $\mathcal B$ closed under direct limits. Then $\mathcal A = \varinjlim (\mathcal A\cap\modR)$.

\medskip

The term `Telescope Conjecture' is used here because the particular case of TCMC when $R$ is a self-injective artin algebra and $(\mathcal A, \mathcal B)$ is a projective cotorsion pair was shown in \cite{KS} to be equivalent to the following telescope conjecture for compactly generated triangulated categories (in this case---for the stable module category over $R$) which originates in works of Bousfield~\cite{Bo} and Ravenel~\cite{Ra} and has been extensively studied by Krause in~\cite{K,K2}:

\medskip

{\bf Telescope Conjecture for Triangulated Categories.} Every smashing localizing subcategory of a compactly generated triangulated category is generated by compact objects.

\medskip

Under some restrictions on homological dimensions of modules in the cotorsion pair $(\mathcal A, \mathcal B)$, TCMC is known to hold. The first author and co-authors showed in~\cite{AST} that the conclusion of TCMC amounts to saying that the given cotorsion pair is of finite type. If all modules in $\mathcal A$ have finite projective dimension, then the cotorsion pair is tilting~\cite{ST2}, hence of finite type~\cite{BS}. If $R$ is a right noetherian ring and $\mathcal B$ consists of modules of finite injective dimension, then $(\mathcal A, \mathcal B)$ is of finite type, too~\cite{AST}. Therefore, TCMC holds true for example for any cotorsion pair over a ring with finite global dimension. Unfortunately, the interesting connection with triangulated categories introduced in~\cite{KS} works for self-injective artin algebras, where the only cotorsion pairs satisfying the former conditions are the trivial ones.

\smallskip

The aim of this paper is twofold. First, we prove the Countable Telescope Conjecture in Theorem~\ref{thm:main}: any cotorsion pair satisfying the hypotheses of TCMC is of countable type---that is, the class $\mathcal B$ is the $\Ext^1$-orthogonal class to the class of all (strongly) countably presented modules from $\mathcal A$. This is a weaker version of TCMC. We will also show that this result easily implies a more direct argument for a large part of the proof that all tilting classes are of finite type ~\cite{BET,BH,ST2,BS}.

The second goal is to systematically analyze analogies between approximation theory for cotorsion pairs and results about localizations in compactly generated triangulated categories. Considerable efforts have been made on both sides. Cotorsion pairs were introduced by Salce in~\cite{Sa} where he noticed a homological connection between special preenvelopes and precovers---or left and right approximation in the terminology of~\cite{ARS}. In~\cite{ET}, Eklof and Trlifaj proved that any cotorsion pair generated by a set of modules provides for these approximations. This turns out to be quite a usual case and the related theory with many applications is explained in the recently issued monograph~\cite{GT}. Localizations of triangulated categories have, on the other hand, motivation in algebraic topology. The telescope conjecture above was introduced by Bousfield~\cite[3.4]{Bo} and Ravenel~\cite[1.33]{Ra}. Compactly generated triangulated categories and their localizations were studied by Neeman~\cite{N1,N2} and Krause~\cite{K,K2}. Even though the telescope conjecture is known to be false for general triangulated categories~\cite{Ke2}, it is still open for the important and topologically motivated stable homotopy category as well as for stable module categories over self-injective artin algebras.

Although it should not be completely unexpected that there are some analogies between the two settings, as the derived unbounded category is triangulated compactly generated and provides a suitable language for homological algebra, the extent to which the analogies work is rather surprising. Roughly speaking, it is sufficient to replace an $\Ext^1$-group in a module category by a $\Hom$-group in a triangulated category, and we obtain a valid result. However, there are also substantial differences here---for instance special precovers and preenvelopes provided by cotorsion pairs are, unlike adjoint functors coming from localizations, not functorial. 

In Section~\ref{sec:def}, we prove in Theorem~\ref{thm:char_B} that if $(\mathcal A,\mathcal B)$ is a cotorsion pair meeting the assumptions of TCMC, then $\mathcal B$ is defined by finite data in the sense that it is the $\Ext^1$-orthogonal class to a certain ideal of maps between finitely presented modules. Moreover, we characterize the countably generated modules in $\mathcal A$ as direct limits of systems of maps from this ideal (Theorem~\ref{thm:char_A}). In Section~\ref{sec:dirlim}, we prove in Theorem~\ref{thm:single_tcmc} that $\mathcal A = \Ker\Ext^1(-,E)$ for a single pure-injective module $E$.

Finally, in Section~\ref{sec:triang}, we give the triangulated category analogues of all of the main results for module categories. Some of them come from our analysis, while the others were originally proved by Krause in~\cite{K} and served as a source of inspiration for this paper.

% - Acknowledgements ------------------------------------------------

\subsection*{Acknowledgements}

The authors would like to thank Jan Trlifaj for reading parts of this text and giving several valuable comments, and also to \O{}yvind Solberg for stimulating discussions and helpful suggestions.

% = Preliminaries ===================================================

\section{Preliminaries}
\label{sec:prelim}

Throughout this paper, $R$ will always stand for an associative ring with unit, and all modules will be (unital) right $R$-modules. We call a module \emph{strongly countably presented} if it has a projective resolution consisting of countably generated projective modules. \emph{Strongly finitely presented} modules are defined in the same manner with the word `countably' replaced by `finitely'. We denote the class of all modules by $\ModR$ and the class of all strongly finitely presented modules by $\modR$.

We note that the notation $\modR$ is often used in the literature for the class of \emph{finitely presented modules}; that is, the modules $M$ possessing a presentation $P_1 \to P_0 \to M \to 0$ where $P_0$ and $P_1$ are finitely generated and projective. We have digressed a little from this de-facto standard for the sake of keeping our notation simple, and we believe that this should not cause much confusion. We remind that if $R$ is a right coherent ring, then the class of strongly finitely presented modules coincides with the class of finitely presented ones. Moreover, one typically restricts oneself to coherent rings in various applications.

% - Continuous directed sets and associated filters -----------------

\subsection{Continuous directed sets and associated filters}

Let $(I, \leq)$ be a partially ordered set and $\lambda$ be an infinite regular cardinal. We say that $I$ is \emph{$\lambda$-complete} if every well-ordered ascending chain $(i_\alpha\mid\alpha<\tau)$ of elements from $I$ of length $<\lambda$ has a supremum in $I$. If this is the case, we call a subset $J \subseteq I$ \emph{$\lambda$-closed} if, whenever such a chain is contained in $J$, its supremum is in $J$ as well.
For instance for any set $X$, the power set $\Pot(X)$ ordered by inclusion is $\lambda$-complete and the set $\Plambda(X)$ of all subsets of $X$ of cardinality $<\lambda$ is $\lambda$-closed in $\Pot(X)$.

Recall that a subset $J \subseteq I$ is called $\emph{cofinal}$ if for every $i \in I$ there is $j \in J$ \st $i\leq j$. Note that if $I$ is a totally ordered set, then the cofinal subsets of $I$ are precisely the unbounded ones.

>From now on, we assume that $(I, \leq)$ is a directed set. If $(M_i, f_{ji}: M_i\to M_j\mid i,j\in I\;\&\; i\leq j)$ is a direct system of modules, we call it \emph{$\lambda$-continuous} if the index set $I$ is $\lambda$-complete and for each well-ordered ascending chain $(i_\alpha\mid\alpha<\tau)$ in $I$ of length $<\lambda$ we have
$$ M_{\sup i_\alpha} = \varinjlim_{\alpha<\tau} M_{i_\alpha}. $$

It is well-known that every module is the direct limit of a direct system of finitely presented modules. But if we want the direct system to be $\lambda$-continuous, we have to pass to $<\lambda$-presented modules in general. The following lemma is a slight modification of \cite[Proposition 7.15]{JL}.

\begin{lem} \label{lem:lambda_dirsys}
Let $M$ be any module and $\lambda$ an infinite regular cardinal. Then $M$ is the direct limit of a $\lambda$-continuous direct system of $<\lambda$-presented modules.
\end{lem}

\begin{proof}
Fix a free presentation
$$ R^{(X)} \overset{f}\to R^{(Y)} \to M \to 0 $$
of $M$ and let $I$ be the following set:
$$ \bigl\{(X',Y')\in \Pot(X) \times \Pot(Y) \mid 
                      \vert X' \vert + \vert Y' \vert < \lambda \; \& \;
                      f\bigl[R^{(X')}\bigr] \subseteq R^{(Y')} \bigr\}. $$

It is straightforward to check that $I$ with the partial ordering by inclusion in both components is directed and $\lambda$-complete. If we now define $M_i$ as the cokernel of the map
$$ f \restriction R^{(X')} : R^{(X')} \to R^{(Y')} $$
for every $i = (X',Y') \in I$, it is easy to check that $(M_i\mid i \in I)$ together with the natural maps forms a $\lambda$-continuous direct system with $M$ as its direct limit.
\end{proof}

For every directed set $I$, there is an \emph{associated filter} $\mathfrak F_I$ on $(\Pot(I),\subseteq)$; namely the one with a basis consisting of the upper sets $\uparrow\! i = \{j \in I \mid j \geq i \}$ for all $i \in I$. That is
$$ \mathfrak F_I = \{X \subseteq I \mid (\exists i \in I)(\uparrow\! i \subseteq X) \}. $$
Recall that a filter $\mathfrak F$ on a power set is called \emph{$\lambda$-complete} if any intersection of less than $\lambda$ elements from $\mathfrak F$ is again in $\mathfrak F$.

\begin{lem} \label{lem:lambda_filt}
Let $(I,\leq)$ be a $\lambda$-complete directed set. Then any subset $J \subseteq I$ \st $\vert J \vert < \lambda$ has an upper bound in $I$. In particular, the associated filter $\mathfrak F_I$ is $\lambda$-complete, and it is a principal filter \iff $(I, \leq)$ has a (unique) maximal element.
\end{lem}

\begin{proof}
We can well-order $J$; that is $J = \{j_\alpha \mid \alpha < \tau\}$ for some $\tau < \lambda$. Then we construct by induction a chain $(k_\alpha\mid\alpha<\tau)$ in $I$ \st $k_0 = j_0$ and $k_\alpha$ is a common upper bound for $j_\alpha$ and $\sup_{\beta<\alpha} k_\beta$. Then $\sup_{\beta<\tau} k_\beta$ is clearly an upper bound for $J$. The rest is also easy.
\end{proof}

% - Filtrations and cotorsion pairs ---------------------------------

\subsection{Filtrations and cotorsion pairs}

% With a similar name but a rather different meaning compared to a filter comes another important concept---a filtration.
Given a module $M$ and an ordinal number $\sigma$, an ascending chain $\mathcal F = (M_\alpha \mid \alpha \leq \sigma)$ of submodules of $M$ is called a \emph{filtration of $M$} if $M_0 = 0$, $M_\sigma = M$ and $\mathcal F$ is \emph{continuous}---that is, $\bigcup _{\alpha < \beta} M_\alpha = M_\beta$ for each limit ordinal $\beta\leq\sigma$.
% (The member $M_\sigma$ is omitted sometimes in the enumeration of $\mathcal F$.)

Furthermore, let a class $\mathcal C\subseteq\ModR$ be given. Then $\mathcal F$ is said to be a \emph{$\mathcal C$-filtration} if it has the extra property that each its consecutive factor $M_{\alpha+1}/M_\alpha$, $\alpha<\sigma$, is isomorphic to a module from $\mathcal C$. A module $M$ is called \emph{$\mathcal C$-filtered} if it admits (at least one) $\mathcal C$-filtration.

\smallskip

Let us turn our attention to cotorsion pairs now. By a \emph{cotorsion pair} in $\ModR$, we mean a pair $(\mathcal A, \mathcal B)$ of classes of right $R$-modules such that $\mathcal A = \text{Ker\,Ext}^1_R(-,\mathcal B)$ and $\mathcal B = \text{Ker\,Ext}^1_R(\mathcal A,-)$. We say that a cotorsion pair $(\mathcal A, \mathcal B)$ is \emph{hereditary} provided that $\mathcal A$ is closed under kernels of epimorphisms or, equivalently, $\mathcal B$ is closed under cokernels of monomorphisms.

If $(\mathcal A, \mathcal B)$ is a cotorsion pair, then the class $\mathcal A$ is always closed under arbitrary direct sums and contains all projective modules. Dually, the class $\mathcal B$ is closed under direct products and it contains all injective modules.
Also, every class of modules $\mathcal C$ determines two distinguished cotorsion pairs---the cotorsion pair \emph{generated} by $\mathcal C$, that is the one with the right-hand class $\mathcal B$ equal to $\Ker\Ext^1_R(\mathcal C,-)$, and dually the cotorsion pair \emph{cogenerated}\footnote{It may cause some confusion that the meaning of the terms \emph{generated} and \emph{cogenerated} is sometimes swapped in the literature. Our terminology follows the monograph~\cite{GT}.}
by $\mathcal C$---the one with the left-hand class $\mathcal A$ equal to Ker\,Ext${}^1_R(-,\mathcal C)$. We say that $(\mathcal A, \mathcal B)$ is of \emph{finite} or \emph{countable} \emph{type} if it is generated by a set of strongly finitely or strongly countably presented modules, respectively.

We say that a cotorsion pair $(\mathcal A, \mathcal B)$ is \emph{complete} if for every module $M\in\ModR$, there is a short exact sequence $0 \to B \to A \to M \to 0$ such that $A \in \mathcal A$ and $B \in \mathcal B$. The map $A \to M$ is then called a \emph{special $\mathcal A$-precover} of $M$. It is well-known that this condition is equivalent to the dual one saying that $\mathcal B$ provides for \emph{special $\mathcal B$-preenvelopes}; thus, for every $M\in\ModR$ there is in this case also a short exact sequence $0 \to M \to B' \to A' \to 0$ with $A'\in \mathcal A$ and $B' \in \mathcal B$.

Finally, a cotorsion pair is said to be \emph{projective} in the sense of~\cite{AB} if it is hereditary, complete, and $\mathcal A\cap\mathcal B$ is precisely the class of all projective modules. It is an easy exercise to prove that $(\mathcal A, \mathcal B)$ is projective \iff it is complete and $\mathcal B$ contains all projective modules and has the ``two out of three'' property---that is: all three modules in a short exact sequence are in $\mathcal B$ provided that two of them are in $\mathcal B$. To conclude the discussion of terminology concerning cotorsion pairs, we recall that projective cotorsion pairs over self-injective artin algebras are (with a slightly different but equivalent definition) called \emph{thick} in~\cite{KS}.

% - Definable classes and coherent functors -------------------------

\subsection{Definable classes and coherent functors}

We will also need the notion of a definable class of modules. First recall that a covariant additive functor from $\ModR$ to the category of abelian groups is called \emph{coherent} if it commutes with arbitrary products and direct limits. The following important characterization was obtained by Crawley-Boevey:

\begin{lem} \label{lem:coherent} \cite[\S 2.1, Lemma 1]{C}
A functor $F: \ModR \to \Ab$ is coherent if and only if it is isomorphic to $\Coker \Hom_R(f,-)$ for some homomorphism $f: X \to Y$ between  finitely presented modules $X$ and $Y$.
\end{lem}

A class $\mathcal C \subseteq \ModR$ is called \emph{definable} if it satisfies one of the following three equivalent conditions:
\begin{enumerate}
\item $\mathcal C$ is closed under taking arbitrary products, direct limits, and pure submodules;
\item $\mathcal C$ is defined by vanishing of some set of coherent functors;
\item $\mathcal C$ is defined in the first order language of $R$-modules by satisfying some implications $\varphi(\bar x) \to \psi(\bar x)$ where $\varphi(\bar x)$ and $\psi(\bar x)$ are primitive positive formulas.
\end{enumerate}
Primitive positive formulas (pp-formulas for short) are first-order language formulas of the form $(\exists \bar y)(\bar x A = \bar y B)$ for some matrices $A,B$ over $R$. For this paper, the most important consequence of $(3)$ is that definable classes are closed under taking elementarily equivalent modules since they are definable in the first-order language. This in particular implies the well-known fact that a definable class is determined by the pure-injective modules it contains since any module is elementarily equivalent to its pure-injective hull. For equivalence between the three definitions and more details, we refer to~\cite{P}, \cite[\S 2.3]{C}, and~\cite[Section 1]{Z}.

% - Inverse limits and Mittag-Leffler systems -----------------------

\subsection{Inverse limits and the Mittag-Leffler condition}

The computation of $\Ext$ groups can sometimes be reduced to the computation of the derived functors of inverse limit. We will recall this here only for countable inverse systems. For more details on the topic see~\cite[\S 3.5]{W}. Let
$$ \dotsb \to H_{n+1} \overset{h_n}\to H_n \to \dotsb \to H_2 \overset{h_1}\to H_1 \overset{h_0}\to H_0 $$
be a countable inverse system of abelian groups---a tower in the terminology of~\cite{W}. Then its inverse limit $\varprojlim H_n$ and the \emph{first derived functor of the inverse limit}, $\varprojlim^1 H_n$, can be computed using the exact sequence
$$
0 \to \varprojlim H_n \to \prod H_n \overset{\Delta}\to \prod H_n \to {\varprojlim}^1 H_n \to 0
$$
where $\Delta((x_n)_{n<\omega}) = (x_n - h_n(x_{n+1}))_{n<\omega}$. The first derived functor is closely related to the fact that inverse limit is not exact---it is only left exact. Using the exact sequence above and the snake lemma, one easily observes that, given a countable inverse system of short exact sequences $0 \to H_n \to K_n \to L_n \to 0$, there is a canonical long exact sequence
$$
0 \to \varprojlim H_n \to \varprojlim K_n \to \varprojlim L_n \to
{\varprojlim}^1 H_n \to {\varprojlim}^1 K_n \to {\varprojlim}^1 L_n \to 0
$$
In particular, $\varprojlim^1$ is right exact on countable inverse systems.

In practice, one is often interested whether or not $\varprojlim^1 H_n = 0$. To decide this can sometimes be tedious, but there is a useful tool---the notion of Mittag-Leffler inverse systems. Given a countable inverse system of abelian groups $(H_n, h_n \mid n<\omega)$ as above, we say that it is \emph{Mittag-Leffler} if for each $n$ the descending chain
$$ H_n \supseteq h_n(H_{n+1}) \supseteq \dotsb \supseteq h_n h_{n+1} \cdots h_{k-1}(H_k) \supseteq \dotsb $$
is stationary. This occurs, for example, if all the maps $h_n$ are onto. The following important result gives a connection to $\varprojlim^1$:

\begin{prop} \label{prop:ML}
Let $(H_n,h_n \mid n<\omega)$ be a countable inverse system of abelian groups. Then the following hold:
\begin{enumerate}
\item \cite[Proposition 3.5.7]{W} If $(H_n,h_n)$ is Mittag-Leffler, then $\varprojlim^1 H_n = 0$.
\item  \cite[Theorem 1.3]{ABH} $(H_n,h_n)$ is Mittag-Leffler \iff $\varprojlim^1 H_n^{(\omega)} = 0$.
\end{enumerate}
\end{prop}

We will also use a related notion of T-nilpotency. We say that $(H_n,h_n)_{n<\omega}$ is {\it T-nilpotent} if for each $n$ there exists $k>n$ \st the composition $H_k\to H_n$ is zero.

% = Filter-closed classes and factorization systems =================

\section{Filter-closed classes and factorization systems}
\label{sec:filter}

We start with analyzing properties of modules lying in $\Ker \Ext^1_R(-,\mathcal G)$ for a class $\mathcal G$ closed under arbitrary direct products and unions of well-ordered chains. We will always assume in this case that $\mathcal G$ is closed under isomorphic images and that $0 \in \mathcal G$, since the trivial module could be viewed as a product of an empty system. As an application to keep in mind, such classes occur as right-hand classes of cotorsion pairs satisfying the hypotheses of TCMC.

\begin{defn} \label{def:f_prod}
Let $\mathfrak F$ be a filter on the power set $\Pot(X)$ for some set $X$, and let $\{M_x\mid x\in X\}$ be a set of modules. Set $M = \prod_{x\in X}M_x$. Then the \emph{$\mathfrak F$-product} $\Sigma_\mathfrak{F} M$ is the submodule of $M$ \st
$$ \Sigma_\mathfrak{F} M = \{m \in M \mid z(m) \in \mathfrak{F} \} $$
where for an element $m = (m_x\mid x \in X) \in M$, we denote by $z(m)$ its zero set $\{x \in X \mid m_x = 0\}$.

The module $M / \Sigma_\mathfrak{F} M$ is then called an \emph{$\mathfrak F$-reduced product}. Note that for $a,b \in M$, we have an equality $\bar a = \bar b$ in the $\mathfrak F$-reduced product \iff $a$ and $b$ agree on a set of indices that is in the filter $\mathfrak F$.

In the case that $M_x = M_y$ for every pair of elements $x,y\in X$, we speak of an \emph{$\mathfrak F$-power} and an \emph{$\mathfrak F$-reduced power} (of the module $M_x$) instead of an $\mathfrak F$-product and an $\mathfrak F$-reduced product, respectively.

Finally, a nonempty class of modules $\mathcal G$ is called \emph{filter-closed}, if it is closed under arbitrary $\mathfrak F$-products (for any set $X$ and an arbitrary filter $\mathfrak F$ on $\Pot (X)$).
\end{defn}

\begin{lem} \label{lem:filter_cl}
Let $\mathcal G$ be a class of modules closed under arbitrary direct products and unions of well-ordered chains. Then $\mathcal G$ is filter-closed.
\end{lem}

\begin{proof}
It is just a matter of straightforward induction to prove that the closure under unions of well-ordered chains implies closure under arbitrary directed unions---see for instance~\cite[Corollary 1.7]{AR} which is easily adapted for unions. Moreover, any $\mathfrak F$-product is just the directed union of products of the modules with indices from the complementary sets to those belonging to $\mathfrak F$.
\end{proof}

In the next few paragraphs, we will show that filter-closedness of $\mathcal G$ forces existence of certain factoring systems inside modules from $\Ker \Ext^1_R(-,\mathcal G)$. Let us note that the following lemma presents the crucial technical step in proving the Countable Telescope Conjecture.

\begin{lem} \label{lem:factor}
Let $\mathcal G$ be a filter-closed class of modules. Let $\lambda$ be an \emph{uncountable} regular cardinal and $(M, f_i\mid i\in I)$ be a direct limit of a $\lambda$-continuous direct system $(M_i, f_{ji}\mid i\leq j)$ indexed by a set $I$ and consisting of $<\lambda$-generated modules.

Assume that $\Ext^1_R(M,\mathcal G) = 0$. Then there is a $\lambda$-closed cofinal subset $J \subseteq I$ \st every homomorphism from $M_j$ to $B$ factors through $f_j$ whenever $j \in J$ and $B \in \mathcal G$.
\end{lem}

\begin{proof}
Suppose that the claim of the lemma is not true. Then the set
$$ S = \{ i \in I \mid
(\exists B_i \in \mathcal G)(\exists g_i\in\Hom _R(M_i, B_i))
(g_i \textrm{ does not factor through } f_i) \}
\eqno{(*)}
$$
must intersect every $\lambda$-closed cofinal subset of $I$ (so $S$ is a generalized stationary set, in an obvious sense). For each $i\in S$, choose some $B_i\in\mathcal G$ and $g_i:M_i\to B_i$ whose existence is claimed in $(*)$. For the indices $i\in I\setminus S$, let $B_i$ be an arbitrary module from $\mathcal G$ and $g_i:M_i \to B_i$ be the zero map. Put $B = \prod_{i \in I} B_i$.

Now, define a homomorphism $h_{ji}: M_i \to B_j$ for each pair $i,j \in I$ in the following way: $h_{ji} = g_j \circ f_{ji}$ if $i \leq j$ and $h_{ji} = 0$ otherwise. This family of maps gives rise to a canonical homomorphism $h: \bigoplus_{k \in I} M_k \to B$. More precisely, if we denote by $\pi_j: B \to B_j$ the projection to the $j$-th component and by $\nu_i: M_i \to \bigoplus _{k\in I}M_k$ the canonical inclusion of the $i$-th component, $h$ is (unique) such that $\pi_j \circ h \circ \nu_i = h_{ji}$. Note that for every $i,j \in I$ \st $i \leq j$, the set $\{k \in I \mid h_{ki} = h_{kj} \circ f_{ji} \}$ is in the associated filter $\mathfrak F_I$ since it contains $\uparrow\! j$. Hence, if we denote by $\varphi$ the canonical pure epimorphism $\bigoplus _{i\in I}M_i\to M = \varinjlim _{i\in I}M_i$ (that is \st $\varphi\circ\nu_i = f_i$ for all $i\in I$), there is a well-defined homomorphism $u$ from $M$ to the $\mathfrak F_I$-reduced product $B/\Sigma _{\mathfrak F_I}B$ making the following diagram commutative ($\rho$ denotes the canonical projection):
$$\begin{CD}
	B @>{\rho}>>	B/\Sigma _{\mathfrak F_I}B @>>> 0\ 	\\
	@A{h}AA			@A{u}AA	\\
	\bigoplus _{i\in I} M_i	@>{\varphi}>>	M @>>> 0.
\end{CD}$$

We have $\Sigma _{\mathfrak F_I}B\in\mathcal G$ since $\mathcal G$ is filter-closed. Hence, using the assumption that $\Ext^1_R(M,\Sigma_{\mathfrak F_I} B) = 0$, we can factorize $u$ through $\rho$ to get some $g \in \Hom_R(M,B)$ \st $u = \rho \circ g$. Since the $M_i$ are all $<\lambda$-generated and $\mathfrak F_I$ is $\lambda$-complete by Lemma~\ref{lem:lambda_filt}, we obtain (for every $i\in I$) that ``$h \circ \nu _i$ coincides with $g \circ \varphi \circ \nu _i = g \circ f_i$ on a set from the filter'', that is:
$$ \{k\in I \mid \pi_k \circ g \circ f_i = \pi_k \circ h \circ \nu _i\} \in \mathfrak F_I. \eqno{(**)}$$

Let us define $J$ as follows:
$$ J = \{i \in I \mid (\forall k \geq i)(\pi_k \circ g \circ f_i = g_k \circ f_{ki}) \}.$$
Then clearly, $g_i$ factors through $f_i$ for every $i \in J$ (just by applying the definition of $J$ for $k = i$). Hence certainly $J \cap S = \varnothing$.

To obtain a contradiction and finish the proof of the lemma, it is now enough to show that $J$ is $\lambda$-closed cofinal.
The fact that $J$ is $\lambda$-closed follows easily by $\lambda$-continuity of the direct system $(M_i, f_{ji}\mid i\leq j)$. So we are left to prove that $J$ is cofinal in $I$. But by $(**)$ and the definition of $\mathfrak F_I$, we can find for every $i \in I$ an element $s(i) \in I$ \st $s(i) \geq i$ and
$$ (\forall k \geq s(i)) (\pi_k \circ g \circ f_i = \pi_k \circ h \circ\nu _i). \eqno{(\Delta)}$$
Recall that $\pi_k \circ h \circ \nu _i = h_{ki} = g_k \circ f_{ki}$. Now, if we fix any $i' \in I$, we can define $j_0 = i'$, $j_{n+1} = s(j_n)$ for all $n \geq 0$, and $j = \sup_{n<\omega} j_n$. Then clearly $j \geq i'$, and it is easy to check that $j \in J$ using the $\aleph _1$-continuity of the direct system $(M_i, f_{ji}\mid i\leq j)$.
\end{proof}
An important consequence follows by applying Lemma~\ref{lem:factor} to the case when the class $\mathcal G$ cogenerates every module. This is for instance always the case when $\mathcal G$ is a right-hand class of a cotorsion pair, since then all injective modules are inside $\mathcal G$.

\begin{prop} \label{prop:ext_submod}
Let $\mathcal G$ be a cogenerating filter-closed class of modules. Then for any uncountable regular cardinal $\lambda$ and any module $M$ \st $\Ext^1_R(M,\mathcal G) = 0$, there is a family $\mathcal C_\lambda$ of $<\lambda$-presented submodules of $M$ \st
\begin{enumerate}
\item $\mathcal C_\lambda$ is closed under unions of well-ordered ascending chains of length $<\lambda$,
\item every subset $X \subseteq M$ \st $\vert X\vert < \lambda$ is contained in some $N \in \mathcal C_\lambda$, and
\item $\Ext^1_R(M/N,\mathcal G) = 0$ for every $N \in \mathcal C_\lambda$.
\end{enumerate}
\end{prop}

\begin{proof}
By Lemma~\ref{lem:lambda_dirsys}, there is a $\lambda$-continuous direct system $(M_i, f_{ji}\mid i\leq j)$ of $<\lambda$-presented modules indexed by a set $I$ \st $M$ together with some maps $f_i: M_i \to M$ forms its direct limit. Now, the data $\mathcal G$, $\lambda$, $(M, f_i\mid i\in I)$, $(M_i, f_{ji}\mid i\leq j)$ and $I$ fits exactly to Lemma~\ref{lem:factor}. Hence, there is a $\lambda$-closed cofinal subset $J \subseteq I$ \st for every $j \in J$, every homomorphism from $M_j$ to a module in $\mathcal G$ factors through $f_j$. But the fact that $\mathcal G$ is a cogenerating class implies that $f_j$ is injective. Thus, we can view the modules $M_j$ for $j \in J$ as submodules of $M$, and the maps $f_j$ and $f_{ji}$ as inclusions. Let us define
$$ \mathcal D = \{M_j \mid j \in J \} $$
and let $\overline{\mathcal D}$ be the closure of $\mathcal D$ under unions of well-ordered chains of length $<\lambda$. 
Observe, that $(\mathcal D,\subseteq)$ is a directed poset since $J$ is a cofinal subset of the directed set $I$. Using Lemma~\ref{lem:lambda_filt}, we easily deduce that $\overline{\mathcal D}$ is directed, too. Now, we can view the modules in $\overline{\mathcal D}$ together with inclusions between them as a $\lambda$-continuous direct system indexed by $\overline{\mathcal D}$ itself. Hence, we can apply Lemma~\ref{lem:factor} for the second time to get a $\lambda$-closed cofinal subset $\mathcal C_\lambda$ of $\overline{\mathcal D}$ \st every homomorphism from a module $N \in \mathcal C_\lambda$ to a module in $\mathcal G$ extends to $M$.

The latter property together with the fact that $\Ext^1_R(M, \mathcal G) = 0$ immediately implies $(3)$. The property $(1)$ is just another way to say that $\mathcal C_\lambda$ is $\lambda$-closed in $\overline{\mathcal D}$. For $(2)$, first notice that $\bigcup \mathcal C_\lambda = M$ since $\mathcal C_\lambda$ is cofinal in $\overline{\mathcal D}$. Hence, if $X \subseteq M$ is a subset of cardinality $<\lambda$, there is a subset $\mathcal M \subseteq \mathcal C_\lambda$ of cardinality $<\lambda$ \st every $x \in X$ is contained in some $N' \in \mathcal M$. Finally, Lemma~\ref{lem:lambda_filt} provides us with an upper bound $N \in \mathcal C_\lambda$ for $\mathcal M$, and clearly $X \subseteq N$.
\end{proof}

In Lemma~\ref{lem:factor}, the assumption of $\lambda$ being uncountable is essential. We can, nevertheless, obtain a weaker but important result using the same technique for $\lambda = \omega$ and $(I,\leq) = (\omega,\leq)$. Lemma~\ref{lem:factor2} actually says that, for $B \in\mathcal G$, the inverse system of groups $(\Hom _R(M_i, B), \Hom _R(f_{ji}, B)\mid i\leq j < \omega)$ is Mittag-Leffler, and the stationary indices determined by $s$ are common over all $B\in\mathcal G$. In this terminology, a proof of the lemma is mostly contained in the proof of~\cite[Theorems 2.5 and 3.7]{BH}.

We give a different proof here and we do this for two main reasons: First, the statement about common stationary indices has an important interpretation in the first-order theory of modules and is missing in~\cite{BH}. Second, we show that the Mittag-Leffler property is a part of a common framework which works for both countable and uncountable systems.

\begin{lem} \label{lem:factor2}
Let $\mathcal G$ be a class of modules closed under countable direct sums. Let $(M, f_i\mid i<\omega)$ be a direct limit of a countable direct system $(M_i, f_{ji}\mid i\leq j<\omega)$ consisting of finitely generated modules.

Assume that $\Ext^1_R(M,\mathcal G) = 0$. Then there is a strictly increasing function $s:\omega \to \omega$ such that for each $B \in \mathcal G$, $i < \omega$ and $c: M_i \to B$ the following holds: If $c$ factors through $f_{s(i)i}$, then it factors through $f_{ni}$ for all $n \ge s(i)$.
\end{lem}

\begin{proof}
We will show that it is possible to construct the values $s(i)$ by induction on $i$. Suppose by way of contradiction that there is some $i < \omega$ for which we cannot define $s(i)$. This can only happen if for each $j \ge i$, there is a homomorphism $g_j:M_j \to B_j$ \st $B_j \in \mathcal G$, and $g_j\circ f_{ji}$ does not factor through $f_{ni}$ for some $n>j$. For $j<i$ let $g_j$ be zero maps and $B_j \in \mathcal G$ be arbitrary. Put $B = \prod _{j < \omega} B_j$.

Now, we follow the proof of Lemma~\ref{lem:factor} (with $\omega$ in place of $I$ and $\lambda$) starting with the second paragraph and ending just after the definition of $(**)$. Note that the corresponding notion of $\aleph _0$-completeness is void, $\mathfrak F_\omega$ is the Fr\' echet filter on $\omega$, and the $\mathfrak F_\omega$-product $\Sigma_{\mathfrak F_\omega} B$ is just the direct sum $\bigoplus _{j < \omega} B_j$.

By the same argument as for $(\Delta)$ in the proof of Lemma~\ref{lem:factor} and with the same notation as there, there is some $s'\ge i$ such that
$$ (\forall k \geq s') (\pi_k \circ g \circ f_i = \pi_k \circ h \circ\nu _i) $$
holds and $\pi_k \circ h \circ \nu _i = h_{ki} = g_k \circ f_{ki}$ for each $k \geq s'$. But this contradicts the fact implied by the choice of $g_k$ that $g_k\circ f_{ki}$ does not factor through $f_i$.
\end{proof}

Let us remark that we have actually proved a little more than we stated in Lemma~\ref{lem:factor2}---we have constructed $s: \omega \to \omega$ \st if $c: M_i \to B$ factors through $f_{s(i)i}$, then it factors through $f_i: M_i \to M$. The motivation for the seemingly more complicated statement of the lemma should become clear in the following paragraphs.

If the modules $M_i$ in the direct system from the lemma above are finitely presented instead of finitely generated, we have a statement about factorization through maps between finitely presented modules. Which in other words means that some coherent functors vanish and the Mittag-Leffler property is preserved within the smallest definable class containing $\mathcal G$. This is made precise by the following lemma.

\begin{lem} \label{lem:factor_def}
Let $\mathcal G$ be a class of modules closed under countable direct sums and $\mathcal D$ be the smallest definable class containing $\mathcal G$. Let $(M, f_i\mid i<\omega)$ be a direct limit of a direct system $(M_i, f_{ji}\mid i\leq j<\omega)$ consisting of finitely presented modules.

Assume that $\Ext^1_R(M,\mathcal G) = 0$. Then there is a strictly increasing function $s:\omega \to \omega$ such that for each $D \in \mathcal D$, $i < \omega$ and $c: M_i \to D$ the following holds: If $c$ factors through $f_{s(i)i}$, then it factors through $f_{ni}$ for all $n \ge s(i)$.
\end{lem}

\begin{proof}
By restating the conclusion of Lemma~\ref{lem:factor2}, we get that $\Img \Hom_R(f_{s(i)i},D) = \Img \Hom_R(f_{ni},D)$ for each $D \in \mathcal G$ and $i \leq s(i) \leq n < \omega$. It is also straightforward to check that $F = \Img \Hom_R(f_{s(i)i},-) / \Img \Hom_R(f_{ni},-)$ is a coherent functor. Hence we have $\Img \Hom_R(f_{s(i)i},D) = \Img \Hom_R(f_{ni},D)$ also for each $D \in \mathcal D$ and the claim follows.
\end{proof}

Note also that instead of vanishing of the coherent functors in the proof above, we can equivalently consider that certain implications between pp-formulas are satisfied~\cite[\S 2.1]{C}, thus reformulating the proof in a more model theoretic way.

Now, we can prove a crucial statement similar to~\cite[Theorem 2.5]{BH}:

\begin{prop} \label{prop:pure_products} Let $\mathcal G$ be a class of modules closed under countable direct sums, and let $M$ be a countably presented module \st $\Ext^1_R(M,\mathcal G) = 0$. Then $\Ext^1_R(M,D) = 0$ for every $D$ isomorphic to a pure submodule of a product of modules from $\mathcal G$.
\end{prop}

\begin{proof} Let $D$ be a pure submodule of $\prod_k B_k$ for some $B_k \in \mathcal G$. Since $M$ is countably presented, it can be considered as a direct limit of a countable chain of finitely presented modules $M_i, i<\omega$, as in the assumptions of Lemma~\ref{lem:factor_def}. Hence $(\Hom_R(M_i, D), \Hom_R(f_{ji}, D)\mid i\leq j<\omega)$ is Mittag-Leffler since any definable class is closed under taking products and pure submodules.

Then we continue as in the proof of~\cite[Theorem 2.5]{BH}. Since $\Ext^1_R\bigl(M,\prod_k B_k\bigr) = 0$ by assumption, we have the exact sequence
$$
\Hom_R\bigl(M,\prod_k B_k\bigr) \overset{h}\to \Hom_R\bigl(M,\bigl(\prod_k B_k\bigr)/D\bigr) \to \Ext^1_R(M,D) \to 0,
$$
and so it suffices to show that $h$ is an epimorphism. This easily follows from Proposition~\ref{prop:ML} applied on the inverse system $(\Hom_R(M_i, D), \Hom_R(f_{ji}, D)\mid i\leq j<\omega)$. Indeed, we see that $\varprojlim^1_i \Hom_R(M_i,D) = 0$ and obtain the exact sequence
$$
\varprojlim_i\Hom_R\bigl(M_i,\prod_k B_k\bigr) \to \varprojlim_i\Hom_R\bigl(M_i,\bigl(\prod_k B_k\bigr)/D\bigr) \to 0.
$$
It remains to use the basic fact that contravariant Hom-functors take colimits to limits.
\end{proof}

% = The countable type ==============================================

\section{Countable type}
\label{sec:cnt}

In this section, we prove the main result of our paper---the Countable Telescope Conjecture for Module Categories. But before doing this, we introduce a fairly simplified version of Shelah's Singular Compactness Theorem. It is based on~\cite[Theorem~IV.3.7]{EM}. In the terminology there, systems witnessing strong $\lambda$-``freeness'' correspond to the $\lambda$-dense systems defined below.

A reader acquainted with the full-fledged compactness theorem for filtrations of modules proved in~\cite[XII.1.14 and IV.3.7]{EM} or~\cite{E} may well skip Lemma~\ref{lem:SSC}.  We state and prove the lemma for the sake of completeness, and also because we are using only a fragment of the full compactness theorem, and it makes the proof of the Countable Telescope Conjecture more transparent.

\begin{defn} \label{def:filtr}
Let $M$ be a module and $\lambda$ be a regular uncountable cardinal. Then a set $\mathcal C_\lambda$ of $<\lambda$-generated submodules of $M$ is called a \emph{$\lambda$-dense system in $M$} if
\begin{enumerate}
\item $0 \in \mathcal C_\lambda$,
\item $\mathcal C_\lambda$ is closed under unions of well-ordered ascending chains of length $<\lambda$, and
\item every subset $X \subseteq M$ \st $\vert X\vert < \lambda$ is contained in some $N \in \mathcal C_\lambda$.
\end{enumerate}
\end{defn}

\begin{lem}[Simplified Shelah's Singular Compactness Theorem] \label{lem:SSC}
Let $\kappa$ be a singular cardinal, $M$ a $\kappa$-generated module, and let $\mu$ be a cardinal \st $\cf \kappa \leq \mu < \kappa$. Suppose we are given a $\lambda$-dense system, $\mathcal C_\lambda$, in $M$ for each regular $\lambda$ \st $\mu < \lambda < \kappa$. Then there is a filtration $(M_\alpha \mid \alpha\leq\cf \kappa)$ of $M$ and a continuous strictly increasing chain of cardinals $(\kappa_\alpha \mid \alpha<\cf \kappa)$ cofinal in $\kappa$ \st $M_\alpha \in \mathcal C_{\kappa_\alpha^+}$ for each $\alpha<\cf \kappa$.
\end{lem}

\begin{proof}
We will start with choosing the chain $(\kappa_\alpha \mid \alpha<\cf \kappa)$. In fact, we can choose any such chain provided that $\mu \leq \kappa_0$, just to make sure that $\mathcal C_{\kappa_\alpha^+}$ is always available. Let us fix one such chain $(\kappa_\alpha \mid \alpha<\cf \kappa)$.

Next, let $(X_\alpha \mid \alpha<\cf \kappa)$ be an ascending chain of subsets of $M$ \st $\bigcup_{\alpha<\cf \kappa} X_\alpha$ generates $M$ and $\vert X_\alpha \vert = \kappa_\alpha$ for each $\alpha<\cf \kappa$. Then, we can by induction construct a (not necessarily continuous) chain $(N^0_\alpha \mid \alpha<\cf \kappa)$ of submodules of $M$ \st $N^0_\alpha \in \mathcal C_{\kappa_\alpha^+}$ and $X_\alpha \cup \bigcup_{\beta<\alpha} N^0_\beta \subseteq N^0_\alpha$ for every $\alpha<\cf \kappa$. Since $N_\alpha$ is $\kappa_\alpha$-generated, we can fix for each $\alpha$ a generating set $Y^0_\alpha$ of $N^0_\alpha$ together with some enumeration $Y^0_\alpha = \{y^0_{\alpha,\gamma} \mid \gamma<\kappa_\alpha\}$. Next, we proceed by induction on $n<\omega$ and construct for each $n>0$ chain of modules $(N^n_\alpha \mid \alpha < \cf \kappa)$ and sets $Y^n_\alpha = \{y^n_{\alpha,\gamma} \mid \gamma<\kappa_\alpha\}$ \st
\begin{enumerate}
\item $(N^n_\alpha \mid \alpha < \cf \kappa)$ is a (not necessarily continuous) chain of submodules of $M$,
\item $N^n_\alpha \in \mathcal C_{\kappa_\alpha^+}$ and $N^n_\alpha \supseteq \{y^{n-1}_{\zeta,\gamma} \mid \alpha\leq\zeta<\cf \kappa \;\&\; \gamma<\kappa_\alpha\} \cup \bigcup_{\beta<\alpha} N^n_\beta$, and
\item $Y^n_\alpha = \{y^n_{\alpha,\gamma} \mid \gamma<\kappa_\alpha\}$ is a fixed enumeration of some set of generators of $N^n_\alpha$, for each $\alpha<\cf \kappa$.
\end{enumerate}
For each $n<\omega$, we clearly can construct such a chain and sets by induction on $\alpha$. Note in particular that we have always $N^{n-1}_\alpha \subseteq N^n_\alpha$ since $Y^{n-1}_\alpha = \{y^{n-1}_{\alpha,\gamma} \mid \gamma<\kappa_\alpha\} \subseteq N^n_\alpha$ by $(2)$. Hence, if we define $M_\alpha = \bigcup_{n<\omega} N^n_\alpha$, we clearly have $M_\alpha \in \mathcal C_{\kappa_\alpha^+}$ for each $\alpha<\cf \kappa$. Also, $\bigcup_{\alpha<\cf \kappa} M_\alpha = M$ since $X_\alpha \subseteq N^0_\alpha \subseteq M_\alpha$ for each $\alpha$. We claim that the chain $(M_\alpha \mid \alpha<\cf \kappa)$ is continuous. To see this, fix for this moment a limit ordinal $\alpha<\cf \kappa$. Then clearly $M_\alpha \supseteq \bigcup_{\beta<\alpha} M_\beta$. On the other hand, for a given $n>0$ and $\beta<\alpha$, we have $\{y^{n-1}_{\alpha,\gamma}\ \mid \gamma<\kappa_\beta\} \subseteq N^n_\beta$ by $(2)$. Therefore, $Y^{n-1}_\alpha \subseteq \bigcup_{\beta<\alpha} N^n_\beta$ and also $N^{n-1}_\alpha \subseteq \bigcup_{\beta<\alpha} N^n_\beta$ by $(3)$. Hence $M_\alpha \subseteq \bigcup_{\beta<\alpha} M_\beta$ and the claim is proved. Now, if we change $M_0$ for the zero module and put $M_{\cf \kappa} = M$, $(M_\alpha \mid \alpha\leq\cf \kappa)$ becomes a filtration with the desired properties.
\end{proof}

While Lemma~\ref{lem:SSC} or Shelah's Singular Compactness Theorem give us some information about the structure of a module with enough dense systems for a singular number of generators, we can prove a rather straightforward lemma which takes care of regular cardinals.

\begin{lem} \label{lem:reg_card}
Let $\kappa$ be a regular uncountable cardinal, $M$ be a $\kappa$-generated module and $\mathcal C_\kappa$ be a $\kappa$-dense system in $M$. Then there is a filtration $(M_\alpha \mid \alpha\leq\kappa)$ of $M$ \st $M_\alpha \in \mathcal C_\kappa$ for each $\alpha<\kappa$.
\end{lem}

\begin{proof}
Let us fix an enumeration $\{m_\gamma \mid \gamma<\kappa\}$ of generators of $M$. We will construct the filtration by induction. Put $M_0 = 0$ and $M_\alpha = \bigcup_{\beta<\alpha} M_\beta$ for all limit ordinals $\alpha\leq\kappa$. For $\alpha = \beta+1$, we can find $M_\alpha\in\mathcal C_\kappa$ such that $M_\beta\cup\{m_\beta\}\subseteq M_\alpha$, using $(3)$ from Definition~\ref{def:filtr}.
\end{proof}

Before stating and proving the main result, we need a technical lemma about filtrations which has been studied in~\cite{FL,SaT,ST}, and whose origins can be traced back to an ingenious idea of P.\ Hill~\cite{Hi}.

\begin{lem} \cite[Theorem 6]{ST}. \label{lem:hill}
Let $\mathcal S$ be a set of countably presented modules and $M$ be a module possessing an $\mathcal S$-filtration $(M_\alpha \mid \alpha\leq\sigma)$. Then there is a family $\mathcal F$ of submodules of $M$ such that:
\begin{enumerate}
\item $M_\alpha \in \mathcal F$ for all $\alpha\leq\sigma$.
\item $\mathcal F$ is closed under arbitrary sums and intersections.
\item For each $N,P \in \mathcal F$ \st $N \subseteq P$, the module $P/N$ is $\mathcal S$-filtered.
\item For each $N \in \mathcal F$ and a countable subset $X \subseteq M$, there is $P \in \mathcal F$ \st $N \cup X \subseteq P$ and $P/N$ is countably presented.
\end{enumerate}
\end{lem}

Now, we are in a position to prove the Countable Telescope Conjecture.

\begin{thm}[Countable Telescope Conjecture] \label{thm:main}
Let $R$ be a ring and $\mathfrak C = (\mathcal A,\mathcal B)$ be a hereditary cotorsion pair of $R$-modules \st $\mathcal B$ is closed under unions of well-ordered chains. Then
\begin{enumerate}
\item $\mathfrak C$ is generated by a set of strongly countably presented modules,
\item $\mathfrak C$ is complete, and
\item $\mathcal B$ is a definable class.
\end{enumerate}
\end{thm}

\begin{proof}
$(1)$. First, we claim that $\mathfrak C$ is generated by a representative set $\mathcal S$ of the class of all countably presented modules from $\mathcal A$. To do this, in view of Eklof's Lemma (\cite[Lemma 3.1.2]{GT} or \cite[Lemma 1]{ET}), it is enough to prove that every module $M \in \mathcal A$ has an $\mathcal S$-filtration $(M_\alpha \mid \alpha\leq\sigma)$.

We will prove this by induction on the minimal cardinal $\kappa$ \st $M$ is $\kappa$-presented. If $\kappa$ is finite or countable, then we are done since $M$ itself is isomorphic to a module from $\mathcal S$. Assume that $\kappa$ is uncountable. By our assumption and Lemma~\ref{lem:filter_cl}, the class $\mathcal B$ is filter-closed and cogenerating. Hence, we can fix for each regular uncountable $\lambda\leq\kappa$ a family $\mathcal C_\lambda$ of $<\lambda$-presented modules given by Proposition~\ref{prop:ext_submod} used with $\mathcal G=\mathcal B$. Note that we can without loss of generality assume that $\mathcal C_\lambda$ is a $\lambda$-dense system, since we always can add the zero module to $\mathcal C_\lambda$ without changing its properties. Then, we can use Lemma~\ref{lem:reg_card} if $\kappa$ is regular, and Lemma~\ref{lem:SSC} if $\kappa$ is singular to obtain a filtration $(L_\beta \mid \beta\leq\tau)$ of $M$ \st for each $\beta<\tau$
\begin{enumerate}
\item[(i)] $L_\beta$ is $<\kappa$-presented, and
\item[(ii)] $M/L_\beta \in \mathcal A$.
\end{enumerate}
We also have $L_{\beta + 1}/L_\beta \in \mathcal A$ since it is a kernel of the projection $M/L_\beta \to M/L_{\beta + 1}$ and $\mathfrak C$ is hereditary. Thus, each of the modules $L_{\beta+1}/L_\beta$ has an $\mathcal S$-filtration by the inductive hypothesis, so we can refine the filtration $(L_\beta \mid \beta\leq\tau)$ to an $\mathcal S$-filtration $(M_\alpha \mid \alpha\leq\sigma)$ of $M$ and the claim is proved. 

Let us note that for the induction step at singular cardinals $\kappa$, we can alternatively use the full version of Shelah's Singular Compactness Theorem, considering $\mathcal S$-filtered modules as ``free'' (cf.~\cite[XII.1.14 and IV.3.7]{EM} or~\cite{E}).

It is still left to show that all modules in $\mathcal S$ are actually strongly countably presented. Note that it is enough to prove that every countably generated module $M \in \mathcal A$ is countably presented. If we prove this, we can take for every module $N \in \mathcal S$ a presentation $0 \to K \to R^{(\omega)} \to N \to 0$ with $K$ a countably generated module. Since $\mathfrak C$ is hereditary, we have $K \in \mathcal A$. Now, if $K$ is countably presented, it must be isomorphic to a module from $\mathcal S$ again, and we can proceed by induction to construct a free resolution of $N$ consisting of countably generated free modules.

So fix $M \in \mathcal A$ countably generated. Then $M$ is $\mathcal S$-filtered by the arguments above. Hence, we can consider the family $\mathcal F$ given by Lemma~\ref{lem:hill} for $M$. To finish our proof, we use $(4)$ from this lemma with $N = 0$ and $X$ a countable set of generators of $M$ as parameters.
\smallskip

$(2)$. This follows from $(1)$ by \cite[Theorem 3.2.1]{GT}.

\smallskip

$(3)$. Note that $\mathcal B$ is always closed under arbitrary direct products. It is closed under infinite direct sums too since these are precisely $\mathfrak F$-products corresponding to Fr\' echet filters $\mathfrak F$. Then $\mathcal B$ is closed under pure submodules by $(1)$ and Proposition~\ref{prop:pure_products}. Further, $\mathcal B$ is closed under pure epimorphic images and, therefore, also under arbitrary direct limits since $\mathfrak C$ is hereditary. Hence $\mathcal B$ is definable.
\end{proof}

\begin{rem}
We can actually prove a little more than we state in Theorem~\ref{thm:main}. Notice that the proof of $(1)$ and $(2)$ works also for any hereditary cotorsion pair cogenerated (as a cotorsion pair) by some cogenerating (in the module category) filter-closed class $\mathcal G$.
\end{rem}

% - Application to tilting classes ----------------------------------

To conclude this section, we will discuss the relation of Theorem~\ref{thm:main} to tilting theory. In fact, it turns out that the countable type and definability of tilting classes is a rather easy consequence of Theorem~\ref{thm:main}. This allows us to give a more direct argumentation for most of the proof of the fact that all tilting classes are of finite type \cite{BH,BS}.

Recall that $\mathfrak T = (\mathcal A, \mathcal B)$ is called a \emph{tilting cotorsion pair} if $\mathfrak T$ is hereditary, $\mathcal A$ consists of modules of finite projective dimension, and $\mathcal B$ is closed under direct sums. In this case, $\mathcal B$ is said to be a \emph{tilting class}.

\begin{thm}\label{thm:tilting} Let $R$ be a ring and $\mathfrak T = (\mathcal A, \mathcal B)$ be a tilting cotorsion pair. Then $\mathfrak T$ is generated by a set of strongly countably presented modules and $\mathcal B$ is definable.
\end{thm}

\begin{proof} Notice that since $\mathcal A$ is closed under direct sums, there is $n < \omega$ \st projective dimension of any module from $\mathcal A$ is at most $n$. We will prove the theorem by induction on this $n$.

If the $n=0$, then $\mathcal B =$ Mod-$R$ and the statement follows trivially. Let $n>0$. Then it is easy to see that the class $\mathcal D = \Ker\Ext^2_R(\mathcal A,-)$ is tilting and in the corresponding tilting cotorsion pair $(\mathcal C,\mathcal D)$, all modules in $\mathcal C$ have projective dimension $<n$ (cf. \cite[Lemma 4.8]{AST}). Thus $\mathcal D$ is definable by the inductive hypothesis. In particular, it is closed under pure submodules. By a simple dimension shifting argument, one observes that $\mathcal B$ is closed under pure-epimorphic images. Since, by our assumption, $\mathcal B$ is closed under direct sums, it follows that $\mathcal B$ is closed under arbitrary direct limits. Thus we may apply Theorem~\ref{thm:main} to $\mathfrak T$ to finish the proof.
\end{proof}

% = Definability ====================================================

\section{Definability}
\label{sec:def}

In this section, we will give a description of which coherent functors define the class $\mathcal B$ of a cotorsion pair $(\mathcal A,\mathcal B)$ satisfying the hypotheses of TCMC. Our aim is twofold: First, vanishing of a coherent functor on a module $M$ translates to the fact that a certain implication between pp-formulas is satisfied in $M$, \cite[\S 2.1]{C}. So there is a clear model-theoretic motivation. Second, proving that the cotorsion pair is of finite type amounts to showing that $\mathcal B$ is defined by a family of coherent functors of the form $\Coker \Hom_R(f,-)$ where $f: X \to Y$ is an inclusion of $X\in\modR$ into a finitely generated projective module $Y$. The projectivity of $Y$ is essential here: it implies that $Y\in\mathcal A$ which in turn means that the functor $\Coker \Hom_R(f,-)$ vanishes on all modules from $\mathcal B$ if and only if $Y/X \in\mathcal A$. Compare this with Remark $(ii)$ at the end of the section.

Even though the finite type question still remains open, we will describe a family of coherent functors defining $\mathcal B$ in Theorem~\ref{thm:char_B}---this can be viewed as a counterpart of~\cite[Theorem A (3)]{K} for module categories. We will also characterize the countably presented modules from the class $\mathcal A$ in Theorem~\ref{thm:char_A}. In both tasks, the key role is played by the ideal $\mathfrak I$ of the category $\modR$ consisting of the morphisms which, when considered in $\ModR$, factor through some module from $\mathcal A$.

For the whole section, let $R$ be a \emph{right coherent} ring; that is, finitely (and also countably) presented modules are precisely the strongly finitely (countably) presented ones, respectively. We will deal with countable direct systems of finitely generated modules of the form:
$$C_0 \overset{f_0}\to C_1 \overset{f_1}\to C_2 \to \dotsb \to C_n \overset{f_n}\to C_{n+1} \to \dotsb .$$
Here, we write for simplicity $f_n$ instead of $f_{n+1,n}$. We start with recalling some important preliminary results whose proofs are essentially in~\cite{BH} and~\cite{ABH}:

\begin{lem} \label{lem:ext_to_lim1}
Let $(C_n, f_n)_{n<\omega}$ be a countable direct system of $R$-modules. Let $M$ be a module \st $\Ext^1_R(\varinjlim C_n,M) = 0$. Then $\varprojlim^1 \Hom_R(C_n,M) = 0$.
\end{lem}

\begin{proof}
The proof here is in fact a part of the proof of \cite[Theorem 5.1]{BH}. If we apply the functor $\Hom_R(-,M)$ to the canonical presentation 
$$
0 \to \bigoplus C_n \overset{\phi}\to \bigoplus C_n \to \varinjlim C_n \to 0
$$
of the countable direct limit $\varinjlim C_n$, we get exactly the first three terms of the exact sequence defining the first derived functor of inverse limit of the system $(H_n \mid n<\omega)$, where $H_n = \Hom_R(C_n,M)$:
$$
0 \to \varprojlim H_n \to \prod H_n \overset{\Delta}\to \prod H_n \to {\varprojlim}^1 H_n \to 0
$$
Since $\Ext^1_R(\varinjlim C_n,M) = 0$, the map $\Delta = \Hom_R(\phi,M)$ is surjective. Hence $\varprojlim^1 H_n = 0$.
\end{proof}

\begin{cor} \label{cor:ext_to_ML}
Let $(C_n, f_n)_{n<\omega}$ be a countable direct system of finitely generated modules. Let $M$ be a module \st $\Ext^1_R(\varinjlim C_n,M^{(\omega)}) = 0$. Then the inverse system $(\Hom_R(C_n,M), \Hom_R(f_n,M))_{n<\omega}$ is Mittag-Leffler.
\end{cor}

\begin{proof}
This follows either immediately from Lemma~\ref{lem:factor2} for $\mathcal G = \{N \mid N \cong M^{(\omega)}\}$, or from Proposition~\ref{prop:ML}. Note that in both cases we use the fact that all modules $C_n$ are finitely generated.
\end{proof}

The following lemma gives us information about a syzygy of a countable direct limit of finitely presented modules and it will be useful for computation.

\begin{lem} \label{lem:horseshoe}
Let $(C_n, f_n)_{n<\omega}$ be a countable direct system of finitely presented modules. Then there exists a countable direct system
$$
\begin{CD}
	@.		\vdots	@.		\vdots	@.		\vdots			\\
@.			@AAA			@AAA			@AAA			\\
0	@>>>	D_2	@>{i_2}>>	P_2	@>{p_2}>>	C_2	@>>>	0	\\
@.			@A{g_1}AA		@A{s_1}AA		@A{f_1}AA		\\
0	@>>>	D_1	@>{i_1}>>	P_1	@>{p_1}>>	C_1	@>>>	0	\\
@.			@A{g_0}AA		@A{s_0}AA		@A{f_0}AA		\\
0	@>>>	D_0	@>{i_0}>>	P_0	@>{p_0}>>	C_0	@>>>	0
\end{CD}
$$

\smallskip\noindent
of short exact sequences of finitely presented modules \st $P_n$ is projective and $s_n$ is split mono for each $n<\omega$. In particular, $\varinjlim P_n$ is projective.
\end{lem}

\begin{proof}
We will construct the short exact sequences by induction on $n$. For $n=0$, let $0 \to D_0 \overset{i_0}\to P_0 \overset{p_0}\to C_0 \to 0$ be a short exact sequence with $P_0$ projective finitely generated. Then $D_0$ is finitely generated, hence finitely presented since we are working over a right coherent ring. If $0 \to D_n \overset{i_n}\to P_n \overset{p_n}\to C_n \to 0$ has already been constructed, let $q: Q \to C_{n+1}$ be an epimorphism \st $Q$ is a finitely generated projective module. Now define $P_{n+1} = P_n \oplus Q$, $s_n: P_n \to P_{n+1}$ as the canonical inclusion, and $p_{n+1} = (f_n p_n,q)$. Then $D_{n+1} = \Ker p_{n+1}$ is finitely presented and $g_n$ is determined by the commutative diagram above. The last assertion is clear.
\end{proof}

Next, we will need a generalized version of Auslander's well-known lemma. It says that $\Ext^1_R(\varinjlim C_i,M) \cong \varprojlim \Ext^1_R(C_i,M)$ whenever $M$ is a pure-injective module. Note that for a countable direct system $(C_n, f_n)_{n<\omega}$, the fact that $M$ is pure-injective implies that $\varprojlim^1 \Hom_R(C_n,M) = 0$. To see this, we will again use the fact that after applying $\Hom_R(-,M)$ on the canonical pure-exact sequence
$$
0 \to \bigoplus C_i \overset{\phi}\to \bigoplus C_i \to \varinjlim C_i \to 0,
\eqno{(\dag)}
$$
we get first three terms of the exact sequence
$$
0 \to \varprojlim H_n \to \prod H_n \overset{\Delta}\to \prod H_n \to {\varprojlim}^1 H_n \to 0
$$
where $H_n = \Hom_R(C_n,M)$. But if $M$ is pure-injective, then applying $\Hom_R(-,M)$ on $(\dag)$ yields an exact sequence and consequently $\varprojlim^1 \Hom_R(C_i,M) = 0$. It turns out that the latter condition is sufficient for $\Ext^1_R(-,M)$ to turn a direct limit into an inverse limit over a right coherent ring:

\begin{lem} \label{lem:gen_auslander}
Let $(C_n, f_n)_{n<\omega}$ be a countable direct system and let $M$ be a module \st $\varprojlim^1 \Hom_R(C_i,M) = 0$. Then $\Ext^1_R(\varinjlim C_i,M) \cong \varprojlim \Ext^1_R(C_i,M)$.
\end{lem}

\begin{proof}
Consider the direct system of short exact sequences $0 \to D_n \overset{i_n}\to P_n \overset{p_n}\to C_n \to 0$ given by Lemma~\ref{lem:horseshoe}. After applying $\Hom_R(-,M)$, we get an inverse system of exact sequences
$$ 0 \to \Hom_R(C_n,M) \overset{p_n^*}\to \Hom_R(P_n,M) \overset{i_n^*}\to \Hom_R(D_n,M) \overset{\delta_n}\to \Ext^1_R(C_n,M) \to 0. $$
By assumption, the following short sequence is exact:
$$ 0 \to \varprojlim \Hom_R(C_n,M) \to \varprojlim \Hom_R(P_n,M) \to \varprojlim \Img i_n^* \to 0. $$

On the other hand, it follows from Proposition~\ref{prop:ML} that $\varprojlim^1 \Hom_R(P_n,M) = 0$ since $(\Hom_R(P_n,M), \Hom_R(s_n,M))_{n<\omega}$ is a countable inverse system with all the maps (split) epic. Moreover, $\varprojlim^1 \Img i_n^* = 0$ since $\varprojlim^1$ is right exact on countable inverse systems. Hence, the following sequence is also exact:
$$ 0 \to \varprojlim \Img i_n^* \to \varprojlim \Hom_R(D_n,M) \to \varprojlim \Ext^1_R(C_n,M) \to 0. $$

Putting everything together, we have obtained the following diagram with canonical maps and exact rows:
$$
\begin{CD}
\varprojlim \Hom_R(P_n,M)	@>>>	\varprojlim \Hom_R(D_n,M)	@>>>	\varprojlim \Ext^1_R(C_n,M)	@>>>	0	\\
@A{\cong}AA							@A{\cong}AA																	\\
\Hom(\varinjlim P_n,M)		@>>>	\Hom(\varinjlim D_n,M)		@>>>	\Ext^1_R(\varinjlim C_n,M)	@>>>	0
\end{CD}
$$
It follows that $\Ext^1_R(\varinjlim C_n,M) \cong \varprojlim \Ext^1_R(C_n,M)$.
\end{proof}

Now, we will focus on T-nilpotent inverse systems. It is clear that every T-nilpotent countable inverse system is Mittag-Leffler. It turns out that the converse is true precisely when the inverse limit of the system vanishes. This is made precise by the following lemma:

\begin{lem} \label{lem:t-nil}
Let $(H_n, h_n)_{n<\omega}$ be a countable inverse system of abelian groups. Then the following are equivalent:
\begin{enumerate}
\item $(H_n, h_n)_{n<\omega}$ is T-nilpotent,
\item $(H_n, h_n)_{n<\omega}$ is Mittag-Leffler and $\varprojlim H_n = 0$.
\end{enumerate}
\end{lem}

\begin{proof}
$(1) \implies (2)$ follows easily from the definitions. Let us prove $(2) \implies (1)$. For each $m < \omega$, let $s(m) > m$ be minimal \st the chain
$$ H_m \supseteq h_m(H_{m+1}) \supseteq \dotsb \supseteq h_m h_{m+1} \cdots h_{n-1}(H_n) \supseteq \dotsb $$
is constant for $n \geq s(m)$ and let $\rho_m: \varprojlim H_n \to H_m$ be the limit map for each $m$. It follows easily that $s(m) \leq s(m')$ for $m < m'$. We will prove by induction that $\Img \rho_m = \Img h_m h_{m+1} \cdots h_{s(m)-1}$. Together with the assumption that $\varprojlim H_n = 0$, this will imply the T-nilpotency. Let us fix $x_m \in \Img h_m h_{m+1} \cdots h_{s(m)-1}$. All we need to do is to construct by induction a sequence of elements $(x_n)_{m<n<\omega}$ \st $x_n \in \Img h_n h_{n+1} \cdots h_{s(n)-1} \subseteq H_n$ and $x_{n-1} = h_{n-1}(x_n)$ for each $n > m$. Suppose we have already constructed $x_{n-1}$ for some $n$. Then, by the chain condition, there is $y \in H_{s(n)}$ \st $h_{n-1} h_n \cdots h_{s(n)-1}(y) = x_{n-1}$. We can put $x_n = h_n \cdots h_{s(n)-1}(y)$.
\end{proof}

We are in a position now to give a connection between vanishing of $\Ext^i_R$ and the chain conditions mentioned above (the Mittag-Leffler condition and T-nilpotency). We state the connection in the following key lemma:

\begin{lem} \label{lem:ext_Tnil}
Let $(C_n, f_n)_{n<\omega}$ be a countable direct system of finitely presented modules and let $M$ be an arbitrary module. Consider the following conditions:
\begin{enumerate}
\item $\Ext^1_R(\varinjlim C_n, M^{(\omega)}) = \Ext^2_R(\varinjlim C_n, M^{(\omega)}) = 0$.
\item The inverse system $(\Hom_R(C_n,M),\Hom_R(f_n,M))_{n<\omega}$ is Mittag-Leffler and $(\Ext^1_R(C_n,M),\Ext^1_R(f_n,M))_{n<\omega}$ is T-nilpotent.
\item $\Ext^1_R(\varinjlim C_n, M^{(\omega)}) = 0$.
\end{enumerate}
Then (1) implies (2) and (2) implies (3).
\end{lem}

\begin{proof}
(1) $\implies$ (2). Assume $\Ext^1_R(\varinjlim C_n, M^{(\omega)}) = \Ext^2_R(\varinjlim C_n, M^{(\omega)}) = 0$. Then the inverse system $(\Hom_R(C_n,M),\Hom_R(f_n,M))_{n<\omega}$ is Mittag-Leffler by Corollary~\ref{cor:ext_to_ML}. By Proposition~\ref{prop:ML} we have $\varprojlim^1 \Hom_R(C_n,M) = 0$, and subsequently it follows by Lemma~\ref{lem:gen_auslander} that
$$ \varprojlim \Ext^1_R(C_n,M) \cong \Ext^1_R(\varinjlim C_n,M) = 0 $$
Next, let $0 \to D_n \to P_n \to C_n \to 0$ be the countable direct system given by Lemma~\ref{lem:horseshoe}. Since
$$ \Ext^1_R(\varinjlim D_n,M^{(\omega)}) = \Ext^2_R(\varinjlim C_n,M^{(\omega)}) = 0 $$
by dimension shifting, the inverse system $(\Hom_R(D_n,M))_{n<\omega}$ is also Mittag-Leffler by Corollary~\ref{cor:ext_to_ML}. Then $(\Ext^1_R(C_n,M))_{n<\omega}$ is Mittag-Leffler as well, since an epimorfic image of a Mittag-Leffler inverse system is Mittag-Leffler again, \cite[Proposition 13.2.1]{G}. Thus, $(\Ext^1_R(C_n,M))_{n<\omega}$ is T-nilpotent by Lemma~\ref{lem:t-nil}.

$(2) \implies (3)$. Clearly, condition (2) implies that $(\Hom_R(C_n,M^{(\omega)}))_{n<\omega}$ is Mittag-Leffler and $(\Ext^1_R(C_n,M^{(\omega)}))_{n<\omega}$ is T-nilpotent. Hence
$$ \Ext^1_R(\varinjlim C_n,M^{(\omega)}) = \varprojlim \Ext^1_R(C_n,M^{(\omega)}) = 0 $$
by Lemmas~\ref{lem:gen_auslander} and~\ref{lem:t-nil}.
\end{proof}

With the previous lemma in mind, a natural question arises when $\Ext^1_R(f,M)$ is a zero map for a homomorphism $f: X \to Y$ between finitely presented modules. It is possible to characterize such maps $f$ when $\Ext^1_R(f,M)=0$ as $M$ runs over all modules in the right-hand class of a complete cotorsion pair. We state this precisely in Lemma~\ref{lem:ext_zero_map}. In view of~\cite{KS}, the lemma can be viewed as a module-theoretic counterpart of~\cite[Lemmas 3.4 (3) and 3.8]{K}.

\begin{lem} \label{lem:ext_zero_map}
Let $(\mathcal A,\mathcal B)$ be a complete cotorsion pair in $\ModR$ and let $f: X \to Y$ be a homomorphism between $R$-modules. Then the following are equivalent:
\begin{enumerate}
\item $\Ext^1_R(f,B) = 0$ for every $B \in \mathcal B$,
\item $f$ factors through some module in $\mathcal A$.
\end{enumerate}
\end{lem}

\begin{proof}
%(1) $\implies$ (2). Let $0 \to B_X \overset{i_X}\to A_X \overset{p_X}\to X \to 0$ and $0 \to B_Y \overset{i_Y}\to A_Y \overset{p_Y}\to Y \to 0$ be special $\mathcal A$-precovers of $X$ and $Y$, respectively. Using the precover property, we can find maps $g: A_X \to A_Y$ and $h: B_X \to B_Y$ \st the following diagram commutes:
%
%$$
%\begin{CD}
%0	@>>>	B_X	@>{i_X}>>	A_X	@>{p_X}>>	X	@>>>	0	\\
%@.			@VV{h}V			@VV{g}V			@VV{f}V			\\
%0	@>>>	B_Y	@>{i_Y}>>	A_Y	@>{p_Y}>>	Y	@>>>	0.
%\end{CD}
%$$
%
%When applying the functor $\Hom_R(-,B_Y)$, we obtain the following commutative diagram with exact rows:
%
%$$
%\begin{CD}
%\Hom_R(A_X, B_Y)	@>>>	\Hom_R(B_X, B_Y)	@>>>	\Ext^1_R(X, B_Y)	@>>>	0	\\
%@AA{\Hom_R(g, B_Y)}A	@AA{\Hom_R(h,B_Y)}A			@AA{\Ext^1_R(f, B_Y)\,=\,0}A		\\
%\Hom_R(A_Y, B_Y)	@>>>	\Hom_R(B_Y, B_Y)	@>>>	\Ext^1_R(Y, B_Y)	@>>>	0.
%\end{CD}
%$$
%
%Comparing the two different ways to map $\ident _{B_Y} \in \Hom_R(B_Y,B_Y)$ to $\Ext^1_R(X,B_Y)$, we deduce that $h \in \Hom_R(B_X,B_Y)$ factors through the inclusion $i_X: B_X \to A_X$; that is, there exists $s: A_X \to B_Y$ such that $s \circ i_X = h$. It is a standard fact that existence of such $s$ is equivalent to existence of a homomorphism $v: X \to A_Y$ \st $p_Y \circ v = f$ (see e.g.~\cite[Lemma B1, Appendix B]{JL} or~\cite[Lemma 2.7]{ABH}). Hence, $f$ factors through $A_Y \in \mathcal A$.

(1) $\implies$ (2). Let $0 \to B \to A \to Y \to 0$ be a special $\mathcal A$-precover of $Y$ and consider the following pull-back diagram:
$$
\begin{CD}
0	@>>>	B	@>>>	Q	@>>>	X	@>>>	0		\\
@.			@|			@VVV		@V{f}VV				\\
0	@>>>	B	@>>>	A	@>>>	Y	@>>>	0
\end{CD}
$$
Then the upper row splits by assumption and $f$ factors through $A$.

(2) $\implies$ (1). This is easy, since the assumption that $f$ factors through some $A \in \mathcal A$ implies that $\Ext^1_R(f,B)$ factors through $\Ext^1_R(A,B)=0$ for each $B \in \mathcal B$.
\end{proof}

Now, we can characterize countably presented modules in the left-hand class of a cotorsion pair satisfying the hypotheses of TCMC. Actually, we state the theorem more generally, for cotorsion pairs satisfying somewhat weaker conditions. Recall that by Theorem~\ref{thm:main}, every cotorsion pair satisfying the hypotheses of TCMC is complete.

\begin{thm} \label{thm:char_A}
Let $R$ be a right coherent ring and $(\mathcal A,\mathcal B)$ be a complete hereditary cotorsion pair with $\mathcal B$ closed under (countable) direct sums. Denote by $\mathfrak I$ the ideal of all morphisms in $\modR$ which factor through some module from $\mathcal A$. Then the following are equivalent for a countably presented module $M$:
\begin{enumerate}
\item $M \in \mathcal A$,
\item $M$ is a direct limit of a countable system $(C_n,f_n)_{n<\omega}$ of finitely presented modules \st $f_n \in \mathfrak I$ for every $n$ and $(\Hom_R(C_n,B),\Hom_R(f_n,B))_{n<\omega}$ is Mittag-Leffler for each $B \in \mathcal B$.
\end{enumerate}
If, in addition, $\mathcal A$ is closed under (countable) direct limits, then these conditions are further equivalent to:
\begin{enumerate}
\item[(3)] $M$ is a direct limit of a countable system $(C_n,f_n)_{n<\omega}$ of finitely presented modules \st $f_n \in \mathfrak I$ for every $n$.
\end{enumerate}
\end{thm}

\begin{proof}
$(1) \implies (2)$. Let us fix (any) countable system $(D_n,g_n)_{n<\omega}$ of finitely presented modules \st $M = \varinjlim D_n$. Assume $M \in \mathcal A$ and $B \in \mathcal B$. Then $B^{(\omega)} \in \mathcal B$ and $\Ext^1_R(\varinjlim D_n, B^{(\omega)}) = \Ext^2_R(\varinjlim D_n, B^{(\omega)}) = 0$ by assumption. So the inverse system $(\Hom_R(D_n,B),\Hom_R(g_n,B))_{n<\omega}$ is Mittag-Leffler and the system $(\Ext^1_R(D_n,B),\Ext^1_R(g_n,B))_{n<\omega}$ is T-nilpotent for each $B \in \mathcal B$ by Lemma~\ref{lem:ext_Tnil}.

Now, we will by induction construct a strictly increasing sequence $n_0 < n_1 < \dotsb$ of natural numbers \st the compositions
$$ f_i = g_{n_{i+1}-1} \dots g_{n_i+1} g_{n_i}: D_{n_i} \to D_{n_{i+1}} $$
satisfy $\Ext^1_R(f_i,B) = 0$ for each $i<\omega$ and $B \in \mathcal B$. Let us start with $n_0 = 0$. For the inductive step, assume that $n_i$ has already been constructed. If there is some $l > n_i$ \st $\Ext^1_R(g_{l-1} \dots g_{n_i+1} g_{n_i},B) = 0$ for each $B \in \mathcal B$, we are done since we can put $n_{i+1} = l$. If this was not the case, there would be some $B_l \in \mathcal B$ for each $l > n_i$ \st $\Ext^1_R(g_{l-1} \dots g_{n_i+1} g_{n_i},B_l) \ne 0$. But this would imply that $(\Ext^1_R(D_n,\bigoplus_{l>n_i} B_l))_{n<\omega}$ is not T-nilpotent, a contradiction.

Finally, we can just put $C_i = D_{n_i}$ and observe using Lemma~\ref{lem:ext_zero_map} that $f_i \in \mathfrak I$ for each $i < \omega$.

$(2) \implies (1)$. This follows directly from Lemma~\ref{lem:ext_Tnil}, since the inverse system $(\Ext^1_R(C_n,B),\Ext^1_R(f_n,B))_{n<\omega}$ is clearly T-nilpotent for each $B \in \mathcal B$ (see Lemma~\ref{lem:ext_zero_map}).

$(2) \implies (3)$ is obvious.

$(3) \implies (1)$. For each $n$, write $f_n$ as a composition of the form $C_n \overset{u_n}\to A_n \overset{v_n}\to C_{n+1}$ with $A_n \in \mathcal A$. In this way, we get a direct system
$$
C_0 \overset{u_0}\to A_0 \overset{v_0}\to C_1 \overset{u_1}\to A_1 \overset{v_1}\to C_2 \overset{u_2}\to \dotsb .$$
Now, $\varinjlim_{n<\omega} C_n = \varinjlim_{n<\omega} A_n$. Hence $M \in \mathcal A$ since $\mathcal A$ is closed under countable direct limits.
\end{proof}

The preceding theorem allows us to characterize modules in the right-hand class of a cotorsion pair satisfying the assumptions of TCMC. Again, we state the following theorem for more general cotorsion pairs than those in question for TCMC. Note that for projective cotorsion pairs over self-injective artin algebras, the following statement is a consequence of~\cite[Corollary 7.7]{KS} and \cite[Theorem A]{K}.

\begin{thm} \label{thm:char_B}
Let $R$ be a right coherent ring and $(\mathcal A,\mathcal B)$ be a hereditary cotorsion pair in $\ModR$ with $\mathcal B$ closed under unions of well-ordered chains. Denote by $\mathfrak I$ the ideal of all morphisms in $\modR$ which factor through some module from $\mathcal A$. Then the following are equivalent:
\begin{enumerate}
\item $B \in \mathcal B$,
\item $\Ext^1_R(f,B) = 0$ for each $f \in \mathfrak I$.
\end{enumerate}
\end{thm}

\begin{proof}
(1) $\implies$ (2). This is clear, since in this case, for each $f \in \mathfrak I$, the map $\Ext^1_R(f,B)$ factors through $\Ext^1_R(A,B)=0$ for some $A \in \mathcal A$.

(2) $\implies$ (1). Recall that the cotorsion pair is of countable type and complete by Theorem~\ref{thm:main}. Moreover, every countably presented module in $\mathcal A$ can be expressed as a direct limit of a direct system $(C_n,f_n)_{n<\omega}$ with all the morphisms $f_n$ in $\mathfrak I$ by Theorem~\ref{thm:char_A}.

Let us define a class of modules $\mathcal C$ as
$$
\mathcal C = \{M \in \ModR \mid \Ext^1_R(f,M) = 0 \textrm{ for each } f \in \mathfrak I \}
$$
%
%It follows immediately that the proof of $(2) \implies (1)$ translates to showing that $\mathcal C \subseteq \mathcal B$.
By definition $\mathcal B \subseteq \mathcal C$.

Note that since every $f \in \mathfrak I$ is a morphism between strongly finitely presented modules, say $f: X \to Y$, and it is not difficult to see that the functors $\Ext^1_R(X,-)$ and $\Ext^1_R(Y,-)$ are coherent in this case, so is the functor $F_f = \Img \Ext^1_R(f,-)$. Hence $\mathcal C$ is a definable class as it is defined by vanishing of the functors $F_f$ where $f$ runs through a representative set of morphisms from $\mathfrak I$. In particular, this means that showing $\mathcal C \subseteq \mathcal B$ reduces just to showing that every \emph{pure-injective} module $M \in \mathcal C$ is already in $\mathcal B$, since definable classes are determined by the pure-injective modules they contain.

To this end, assume that $M \in \mathcal C$ is pure-injective and $A \in \mathcal A$ is countably presented. Then $A = \varinjlim C_n$ where $(C_n,f_n)_{n<\omega}$ is a direct system \st $f_n \in \mathfrak I$ for each $n$. In particular, $\Ext^1_R(f_n,M) = 0$ by assumption and
$$ \Ext^1_R(A,M) = \Ext^1_R(\varinjlim C_n,M) \cong \varprojlim \Ext^1_R(C_n,M) = 0 $$
by Auslander's lemma. Finally, since $(\mathcal A,\mathcal B)$ is of countable type and $A$ was arbitrary, it follows that $M \in \mathcal B$.
\end{proof}

\begin{rem} $(i)$ Countable type of the cotorsion pair considered in Theorem~\ref{thm:char_B} together with Lemma~\ref{lem:hill} imply that when defining $\mathfrak I$, we may assume that the modules from $\mathcal A$ through which the maps $f \in \mathfrak I$ are required to factorize are all countably presented.

\smallskip\noindent
$(ii)$ To determine which implication of pp-formulas corresponds to the coherent functor $F_f$ from the proof of Theorem~\ref{thm:char_B}, we build the following commutative diagram
$$
\begin{CD}
0	@>>>	K	@>{i_X}>>	F_X	@>{p_X}>>	X	@>>>	0	\\
@.			@VV{i}V			@VV{s}V			@VV{f}V			\\
0	@>>>	L	@>{i_Y}>>	F_Y	@>{p_Y}>>	Y	@>>>	0
\end{CD}
$$

\smallskip
\noindent
with $F_X, F_Y$ finitely generated free, $K, L$ finitely presented, $s$ a split embedding and $i, i_X, i_Y$ inclusions. Now, an equivalent statement to $F_f(M) = 0$ is that every homomorphism from $K$ into $M$ which extends to $L$ must extend to $F_X$ as well, and this can be routinely translated to an implication between two pp-formulas to be satisfied in $M$. If we denote by $H$ the pushout of $i$ and $i_X$, and by $h$ the pushout map $L \to H$, then the latter actually means that $\Coker \Hom_R(h,M) = 0$. Thus, $\Coker \Hom _R(h,-)$ is a coherent functor which may be equivalently used instead of $F_f$ when defining $\mathcal B$.
\end{rem}

% = Direct limits and pure-epimorphic images ========================

\section{Direct limits and pure-epimorphic images}
\label{sec:dirlim}

In the cases when TCMC holds true, the class $\mathcal A$ of any cotorsion pair $(\mathcal A, \mathcal B)$ meeting its assumptions must be closed under pure-epimorphic images. Indeed, in this setting, we have $\mathcal A = \varinjlim (\mathcal A\cap\hbox{mod-}R)$ and the latter class is closed under pure-epimorphic images by the well-known result of Lenzing (cf.~\cite{L} or~\cite[Lemma 1.2.9]{GT}). In this section, we prove that the hypotheses of TCMC do always imply that $\mathcal A$ is closed under pure-epimorphic images. As a consequence, we prove that every complete cotorsion pair with both classes closed under arbitrary direct limits is cogenerated by a single pure-injective module---this can be viewed as a module-theoretic counterpart of~\cite[Theorem C]{K}.

Note that the first part---to make sure that $\mathcal A$ is closed under pure-epimorphic images---is the crucial one. For projective cotorsion pairs over self-injective algebras which satisfy the hypotheses of TCMC, this property follows by analysis of the proofs in~\cite{K} and~\cite{KS}. But when proving this in a more general setting, one obstacle appears. Namely, complete cotorsion pairs provide us with approximations (special precovers and preenvelopes) which are not functorial in general. Therefore, implementing the rather simple underlying idea---expressing each module in $\mathcal A$ in terms of direct limits of $\mathcal A$-precovers of finitely presented modules and proving that this transfers to pure-epimorphic images---requires several technical steps. In particular, we need special indexing sets for our direct systems which we call \emph{inverse trees}.

\smallskip

We start with a preparatory lemma. Recall that for an ordinal number $\alpha$, we denote by $|\alpha|$ the cardinality of $\alpha$ when viewed as the set of all smaller ordinals.

\begin{defn} \label{defn:pred}
A direct system $(M_i, f_{ji}\mid i,j\in I\;\&\; i\leq j)$ of $R$-modules is said to be \emph{continuous} if $(M_k, f_{kj}\mid j \in J)$ is the direct limit of the system $(M_i, f_{ji}\mid i,j\in J\;\&\; i\leq j)$ whenever $J$ is a directed subposet of $I$ and $k$ is a supremum of $J$ in $I$.
\end{defn}

\begin{lem} \label{lem:well-chain}
Let $\kappa$ be an infinite cardinal and $M$ be a $\kappa$-presented module. Then $M$ is a direct limit of a continuous well-ordered system $(M_\alpha, f_{\beta\alpha} \mid \alpha\leq\beta<\kappa)$ \st for all $\alpha<\kappa$, $M_\alpha$ is $|\alpha|$-presented.
\end{lem}

\begin{proof}
We can start as in Lemma~\ref{lem:lambda_dirsys}. Let
$$ \bigoplus_{\beta<\kappa} x_\beta R \overset{g}\to \bigoplus_{\gamma<\kappa} y_\gamma R \to M \to 0 $$
be a free presentation of $M$. For each $\alpha < \kappa$, let $X_\alpha$ be the subset of all ordinals $\beta<\alpha$ \st $f(x_\beta) \in \bigoplus_{\gamma<\alpha} y_\gamma R$. If we define $M_\alpha$ as the cokernel of the restriction $\bigoplus_{\beta\in X_\alpha} x_\beta R \to \bigoplus_{\gamma<\alpha} y_\gamma R$ of $g$, it is easy to see that the direct system $(M_\alpha \mid \alpha<\kappa)$ together with the natural maps has the properties we require.
\end{proof}

For a set $X$, we will denote by $X^*$ the set of all finite strings over $X$, that is, all functions $u: n \to X$ for $n<\omega$. We will denote strings by letters $u,v,w,\dots$ and write them as sequences of elements of $X$, which we will denote by Greek letters for a reason which will be clear soon. For example, we write $u = \alpha_0\alpha_1\dots\alpha_{n-1}$. When $u,v$ are strings, we denote by $uv$ their \emph{concatenation}, we define the \emph{length} of a string $u$ in the usual way and denote it by $\ell(u)$, and we identify strings of length $1$ with elements in $X$. The empty string is denoted by $\varnothing$. Note that the set $X^*$ together with the concatenation operation is nothing else than the free monoid over $X$.

\begin{defn} \label{defn:inv_k-tree}
Let $\kappa$ be an infinite cardinal and $\kappa^*$ be the free monoid over $\kappa$. Let us equip $\kappa^*\setminus\{\varnothing\}$ with a partial order in the following way: If $u = \alpha_0\alpha_1\dots\alpha_{n-1}$ and $v = \beta_0\beta_1\dots\beta_{m-1}$, we put $u\leq v$ if
\begin{enumerate}
\item $n \geq m$,
\item $\alpha_0\alpha_1\dots\alpha_{m-2} = \beta_0\beta_1\dots\beta_{m-2}$, and
\item $\alpha_{m-1} \leq \beta_{m-1}$ as ordinal numbers.
\end{enumerate}
Then an \emph{inverse tree} over $\kappa$ is the subposet of $(\kappa^*\setminus\{\varnothing\}, \leq)$ defined as
$$
I_\kappa = \bigl\{\alpha_0\alpha_1\dots\alpha_{n-1}\;\bigl | \bigr.\;\bigl(\forall i \leq n-2\bigr)\bigl(\alpha_i \text{ is infinite, non-limit}\;\&\; \alpha_{i+1}<|\alpha_i|\bigr)\bigr\}.
$$

For convenience, given a non-empty string $u = \alpha_0\alpha_1\dots\alpha_{n-1} \in \kappa^*$, we define the \emph{tail} of $u$, denoted by $t(u)$, to be the last symbol $\alpha_{n-1}$ of $u$, and the \emph{rank} of $u$, $\rk(u)$, to be the cardinal number $|\alpha_{n-1}|$. Notice that in this terminology, the tail of a string $u \in I_\kappa$ is allowed to be a limit or finite ordinal.
\end{defn}

Having defined inverse trees, we can start collecting basic properties of the partial ordering:

\begin{lem} \label{lem:invt_proper}
Let $(I_\kappa, \leq)$ be an inverse tree, and let $v$ and $u = \beta_0\dots\beta_{m-2}\beta_{m-1}$ be two elements of $I_\kappa$ \st $v<u$. Then there is $w \in I_\kappa$ \st $v \leq w < u$ and one of the following cases holds true:
\begin{enumerate}
\item There is an ordinal $\gamma<\beta_{m-1}$ \st $w = \beta_0\beta_1\dots\beta_{m-2}\gamma$.
\item There is an ordinal $\gamma<|\beta_{m-1}|$ \st $w = \beta_0\beta_1\dots\beta_{m-2}\beta_{m-1}\gamma$.
\end{enumerate}
\end{lem}

\begin{proof}
This follows easily from the definition. Notice that $(2)$ can only hold if $\beta_{m-1}=t(u)$ is infinite and non-limit.
\end{proof}

As an immediate corollary, we will see that the properties of $u \in I_\kappa$ with respect to the ordering depend very much on the tail (and rank) of $u$:

\begin{cor} \label{cor:invt_tail}
Let $u = \alpha_0\dots\alpha_{n-2}\alpha_{n-1} \in I_\kappa$. Then the following hold in $(I_\kappa, \leq)$:
\begin{enumerate}
\item If $t(u) = 0$, then $u$ is a minimal element.
\item If $t(u)$ is non-zero finite, then $u$ has a unique immediate predecessor.
\item If $t(u)$ is an infinite non-limit ordinal, then $u = \sup \{u\gamma \mid \gamma<\rk(u)\}$.
\item If $t(u)$ is a limit ordinal, then $u = \sup \{ \alpha_0\dots\alpha_{n-2}\gamma \mid \gamma<t(u)\}$.
\end{enumerate}
\end{cor}

We have seen that an element $u \in I_\kappa$ can be expressed as a supremum of a chain of strictly smaller elements \iff $\rk(u)$ is infinite. If so, this chain depends on whether $t(u)$ is a limit ordinal or not. We will prove in the next lemma that as far as we are concerned with continuous direct systems indexed with $I_\kappa$, this expression of $u$ as a supremum is essentially unique.

\begin{lem} \label{lem:invt_sup}
Let $u \in I_\kappa$ be of infinite rank and $C$ be the chain as in Corollary~\ref{cor:invt_tail} $(3)$ or $(4)$ \st $u = \sup C$ in $I_\kappa$. Let $J \subseteq I_\kappa$ be a directed subposet of $I_\kappa$ \st $u = \sup J$ in $I_\kappa$ and $u \not\in J$. Then $C \cap J$ is cofinal in $J$.
\end{lem}

\begin{proof}
Choose some $j\in J$ of the least possible length. Since $J$ is directed, $u$ is the supremum of the upper set $\uparrow\! j = \{i\in J\mid i\geq j\}$, too. By the definition of the ordering and the fact that $j$ has been taken of the least possible length, we see that each $i\in (\uparrow\! j)$ is of the form $\beta_0\beta_1\dots\beta_{m-2}\gamma_i$ where $\beta_0, \beta_1,\dotsc, \beta_{m-2}$ are fixed and $\gamma_i < |\beta_{m-2}|$. Thus $u = \beta_0\beta_1\dots\beta_{m-2}$ provided that $\sup\{\gamma _i\mid i\in (\uparrow\! j)\} = |\beta_{m-2}|$ (case $(3)$), and $u = \beta_0\beta_1\dots\beta_{m-2}\beta_{m-1}$ if $\beta_{m-1} = \sup\{\gamma _i\mid i\in (\uparrow\! j)\} < |\beta_{m-2}|$ (case $(4)$). Hence, $\uparrow\! j \subseteq C \cap J$ by assumption, and $C \cap J$ is cofinal in $J$ since $\uparrow\! j$ is.
\end{proof}

So far, we have studied elements strictly smaller than a given $u \in I_\kappa$. But, we will also need to look ``upwards'':

\begin{lem} \label{lem:invt_dir}
Let $(I_\kappa, \leq)$ be an inverse tree. Then
\begin{enumerate}
\item For each $u\in I_\kappa$, the upper set $\uparrow\! u = \{w \in I_\kappa \mid w \geq u\}$ is well-ordered.
\item $(I_\kappa, \leq)$ is directed.
\item Every non-empty bounded subset $X \subseteq I_\kappa$ has a supremum in $I_\kappa$.
\end{enumerate}
\end{lem}

\begin{proof}
(1). It follows from the definition that $\uparrow\! u$ is a totally ordered subset of $I_\kappa$. If $X \subseteq (\uparrow\! u)$ is nonempty, then the longest string $u \in X$ with the minimum tail $t(u)$ is the least element in $X$. Hence, $\uparrow\! u$ is well-ordered.

(2). Let $u = \alpha_1 \dots \alpha_{n-1}$, $v = \beta_1 \dots \beta_{m-1}$ be elements in $I_\kappa$. Then $\max \{\alpha_1, \beta_1\}$, viewed as a string of length $1$, is greater than both $u$ and $v$.

(3). Suppose $X \subseteq I_\kappa$ is non-empty and has an upper bound $u \in I_\kappa$. In other words, $u \in Y$ for $Y =  \bigcap_{w \in X} (\uparrow\! w)$. But since for any $v \in X$ clearly $Y \subseteq (\uparrow\! v)$, there must be the least element in $Y$, which is by definition the supremum of $X$.
\end{proof}

In view of the preceding lemma, we can introduce the following definition:

\begin{defn} \label{defn:succ}
Let $(I_\kappa, \leq)$ be an inverse tree and $u = \alpha_0 \dots \alpha_{n-2}\alpha_{n-1} \in I_\kappa$. Then the \emph{successor} of $u$ in $I_\kappa$ is defined as $\textrm{s}(u) = \alpha_0 \dots \alpha_{n-2}\beta$ where $\beta = \alpha+1$ is the ordinal successor of $\alpha$. Similarly, if $t(u) = \alpha_{n-1}$ is non-limit and non-zero, we define the \emph{predecessor} of $u$ as $\textrm{p}(u) = \alpha_0 \dots \alpha_{n-2}\gamma$ where $\gamma = \alpha-1$ is the ordinal predecessor of $\alpha$.
\end{defn}

Note that by Lemma~\ref{lem:invt_dir}, $\textrm{s}(u)$ is the unique immediate successor of $u$ in $(I_\kappa, \leq)$. On the other hand, even if $\textrm{p}(u)$ is defined, there still may be other elements in $I_\kappa$ less than $u$ that are incomparable with $\textrm{p}(u)$---see Lemma~\ref{lem:invt_proper}. We can summarize our observations in a figure showing ``neighbourhoods'' of elements $u \in I_\kappa$ depending on $t(u)$, where $w \in \kappa^*$ is the string obtained from $u$ by removing its last symbol:

\bigskip\noindent
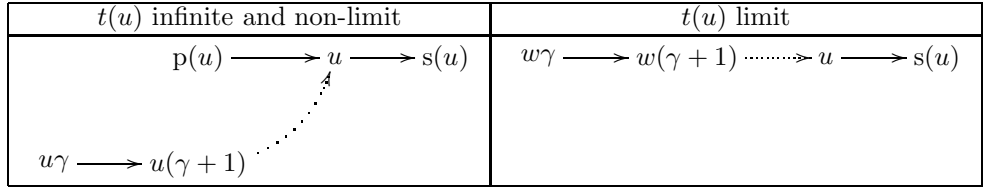
\begin{tabular}{|c|c|}
\hline
$t(u)$ infinite and non-limit
&
$t(u)$ limit
\\ \hline
$$
\xymatrix{
& \textrm{p}(u) \ar[r] & u \ar[r] & \textrm{s}(u) \\
u\gamma \ar[r] & u(\gamma+1) \ar@/_1pc/@{.>}[ur]
}
$$
&
$$
\xymatrix{
w\gamma \ar[r] & w(\gamma+1) \ar@{.>}[r] & u \ar[r] & \textrm{s}(u)
}
$$
\\ \hline
\end{tabular}
\smallskip

This picture also shows the motivation for calling $(I_\kappa, \leq)$ an inverse tree. From each $u \in I_\kappa$, there is exactly one possible way towards greater elements, while when traveling in $I_\kappa$ down the ordering, there are many branches. The rank zero elements of $I_\kappa$ can be viewed as leaves. Just the root is missing---it is easy to see that $I_\kappa$ has no maximal element.

Next, we will turn our attention back to modules. We shall see that each infinitely presented module is the direct limit of a special direct system indexed by an inverse tree.

\begin{lem} \label{lem:inv_k-tree}
Let $\kappa$ be an infinite cardinal and $M$ be a $\kappa$-presented module. Then $M$ is the direct limit of a continuous direct system $(M_u, f_{vu} \mid u,v\in I_\kappa\;\&\; u\leq v)$ indexed by the inverse tree $I_\kappa$ and \st $M_u$ is $\rk(u)$-presented for each $u\in I_\kappa$.
\end{lem}

\begin{proof} We will construct the direct system by induction on $\ell(u)$ using Lemma~\ref{lem:well-chain}. If $\ell(u) = 1$, then $u$ can be viewed as an ordinal number $<\kappa$ and we just use the modules $M_u$ and morphisms $f_{vu}$ obtained for $M$ by Lemma~\ref{lem:well-chain}.

Suppose we have defined $M_u$ and $f_{vu}$ for all $u,v \in I_\kappa$ with $\ell(u), \ell(v)\leq n$. Let $v\in I_\kappa$ be arbitrary with $\ell(v) = n$ and such that $t(v)$ is infinite and non-limit. Then by using Lemma~\ref{lem:well-chain} for $M_v$, we obtain a well-ordered continuous system $(M^v_\alpha, f^v_{\beta\alpha} \mid \alpha\leq\beta<\rk(v))$, and we set $M_{v\alpha} = M^v_\alpha$ and $f_{v\beta,v\alpha} = f^v_{\beta\alpha}$ for all $\alpha\leq\beta<\rk(v)$. Finally, the morphisms $f_{v,v\alpha}$, $\alpha<\rk(v)$, will be defined as the colimit maps $M^v_\alpha \to M_v$, and the rest of the morphisms $f_{u,v\alpha}$ just by taking the appropriate compositions.

The correctness of this construction is ensured by the properties of $I_\kappa$ proved above, and the fact that $(M_u \mid u \in I_\kappa)$ is continuous is taken care of by Lemma~\ref{lem:invt_sup}.
\end{proof}

The crucial fact about inverse trees is that, under the assumptions of TCMC, they allow us to construct for each module a continuous direct system of special precovers:

\begin{lem} \label{lem:precovers}
Let $(\mathcal A, \mathcal B)$ be a complete cotorsion pair with both classes closed under direct limits, $\kappa$ be an infinite cardinal, and $M$ be a $\kappa$-presented module. Then there is a continuous direct system of short exact sequences $0 \to B_u \overset{\iota_u}\to A_u \overset{\pi_u}\to M_u \to 0$ indexed by $I_\kappa$ \st $B_u \in \mathcal B$, $A_u \in \mathcal A$, $M_u$ is $\rk(u)$-presented for each $u \in I_\kappa$, and $M$ is the direct limit of the modules $M_u$.
\end{lem}

\begin{proof}
We start with the continuous direct system $(M_u, f_{vu} \mid u,v\in I_\kappa\;\&\; u\leq v)$ given by Lemma~\ref{lem:inv_k-tree} and construct the exact sequences for each $u \in I_\kappa$ by transfinite induction on $t(u)$.

For each $u\in I_\kappa$ of finite rank, we choose a special $\mathcal A$-precover,
$$ 0 \to B_u \overset{\iota_u}\to A_u \overset{\pi_u}\to M_u \to 0, $$
of $M_u$, and if $t(u) > 0$, we find appropriate morphisms $g_{u\text{p}(u)}: A_{\text{p}(u)} \to A_u$ and $h_{u\text{p}(u)}: B_{\text{p}(u)} \to B_u$ using the precover property for the map $f_{u\text{p}(u)}\circ \pi_{\text{p}(u)}$.

Suppose that $\alpha$ is a limit ordinal and the sequences $0 \to B_u \overset{\iota_u}\to A_u \overset{\pi_u}\to M_u \to 0$ and the maps between them have been constructed for all $u\in I_\kappa$ with $t(u)<\alpha$. Then for each $v\in I_\kappa$ with $t(v)=\alpha$, we define the exact sequence $0 \to B_v \overset{\iota_v}\to A_v \overset{\pi _v} \to M_v \to 0$ as the direct limit of the direct system of already constructed short exact sequences $0 \to B_w \overset{\iota_w}\to A_w \overset{\pi_w}\to M_w \to 0$ where $w$ runs over the chain given by Corollary~\ref{cor:invt_tail} (4) used for $v$. By assumption, we get $A_v\in\mathcal A$ and $B_v\in\mathcal B$.

Finally, suppose that $\alpha = \delta + 1$ for some infinite $\delta$ and we have constructed the exact sequences for all $u \in I_\kappa$ \st $t(u) \leq \delta$.
% If $\beta <\delta$, then we may suppose that $g_u$ and $h_u$ are constructed as well.
Similarly as above, we define for each $v\in I_\kappa$ with $t(v)=\alpha$ the exact sequence $0 \to B_v \overset{\iota_v}\to A_v \overset{\pi _v} \to M_v \to 0$ as the direct limit of the direct system of short exact sequences $0 \to B_{v\beta} \overset{\iota_{v\beta}}\to A_{v\beta} \overset{\pi_{v\beta}}\to M_{v\beta} \to 0$ where $\beta$ runs over all ordinal numbers $<\rk(v)$. The morphisms $g_{v\text{p}(v)}: A_{\text{p}(v)} \to A_v$ and $h_{v\text{p}(v)}: B_{\text{p}(v)} \to B_v$ can be defined again by the precover property and the rest of the morphisms by obvious compositions. This concludes the construction.

The fact that the direct system of the exact sequences just constructed is well-defined and continuous follows from the lemmas above, in particular from Lemmas~\ref{lem:invt_proper} and~\ref{lem:invt_sup}.
\end{proof}

Before stating one of the main results in this section, let us recall that a cotorsion pair satisfying the assumptions of TCMC is complete by Theorem~\ref{thm:main} (2), thus it fits the setting of the following theorem.

\begin{thm} \label{thm:lim->pure}
Let $(\mathcal A, \mathcal B)$ be a complete cotorsion pair with both classes closed under direct limits. Then $\mathcal A$ is closed under pure epimorphic images.
\end{thm}

\begin{proof} Let $M$ be a pure epimorphic image of a module from $\mathcal A$. We can assume that $M$ is not finitely presented since otherwise $M$ is trivially in $\mathcal A$. Hence, Lemma~\ref{lem:precovers} gives us a continuous direct system $0 \to B_u \overset{\iota_u}\to A_u \overset{\pi_u}\to M_u \to 0$ indexed by $I_\kappa$ for some $\kappa$, and the direct limit $0 \to B \overset{\iota}\to A \overset{\pi} \to M \to 0$ of this system is a special $\mathcal A$-precover of $M$. It follows from our assumption on $M$ that $\pi$ is a pure epimorphism.

Now, $M$ is also the direct limit of some direct system $(K_i, k_{ji}\mid i\preceq j)$ consisting of finitely presented modules and indexed by some poset $(J,\preceq)$. We claim that although there is no obvious relation between the direct systems $(M_u \mid u \in I_\kappa)$ and $(K_i \mid i \in J)$, the following holds: For each $i \in J$, there is $s(i) \in J$ \st $i \prec s(i)$ and $k_{s(i)i}$ factors through $A_u$ for some $u \in I_\kappa$ of finite rank.

To this end, denote for all $i\in J$ by $k_i: K_i \to M$ the colimit maps and fix an arbitrary $i\in J$. Then $k_i$ can be factorized through $\pi$ since $K_i$ is finitely presented and $\pi$ is pure. Moreover, since $A = \varinjlim_{I_\kappa} A_u$, there is $u_1 \in I_\kappa$ \st $k_i$ factors through $A_{u_1}$. If $\rk(u_1)$ is finite, we put $u = u_1$. If not, $A_{u_1}$ is by Corollary~\ref{cor:invt_tail} the direct limit of a direct system consisting of some modules $A_v$ with $t(v) < t(u_1)$. Hence, $k_i$ further factors through $A_{u_2}$ for some $u_2 \in I_\kappa$ \st $t(u_2) < t(u_1)$. If the rank of $u_2$ is finite, we put $u = u_2$. Otherwise, we construct in a similar way $u_3$ \st $t(u_3) < t(u_2)$, and so forth. Since there are no infinite descending sequences of ordinals, we must arrive at some $u = u_n$ of finite rank after finitely many steps.

Hence, there must be $u_i \in I_\kappa$ of finite rank \st $k_i$ factors through $\pi \circ g_{u_i} = f_{u_i} \circ \pi_{u_i}$ where $g_{u_i}: A_{u_i} \to A$ and $f_{u_i}: M_{u_i} \to M$ are the colimit maps. That is, $k_i = f_{u_i} \circ \pi_{u_i} \circ e_i$ for some $e_i: K_i \to A_{u_i}$ and, since $M_{u_i}$ is finitely presented by Lemma~\ref{lem:precovers}, $f_{u_i}$ further factors as $k_{j_i} \circ d_{u_i}$ for some $d_{u_i}: M_{u_i} \to K_{j_i}$ and $j_i \in J$ \st $j_i \succ i$. Together, we have $k_i = k_{j_i} \circ d_{u_i} \circ \pi_{u_i} \circ e_i$. Thus, using the fact that $K_i$ is finitely presented and well-known properties of direct limits, there must exist some $s(i)\succeq j_i$ such that $k_{s(i) i} = k_{s(i) j_i} \circ d_{u_i} \circ \pi_{u_i} \circ e_i$, and the claim is proved.

Now set $\tilde J = J \times \{0,1\}$ and define $(\tilde J, \preceq)$ as the poset generated by the relations $(i,0)\preceq (j,0)$ and $(i,0)\preceq (i,1)\preceq (s(i), 0)$ where $i, j\in J, i\preceq j$. Further, for such $i,j$, put $K_{(i,0)}= K_i$, $K_{(i,1)} = A_{u_i}$, $k_{(j,0),(i,0)} = k_{ji}$, $k_{(i,1), (i,0)} = e_i$, and $k_{(s(i), 0),(i,1)} = k_{s(i) j_i} \circ d_{u_i} \circ \pi_{u_i}$, using the same notation as above. In this way, defining the remaining morphisms as the appropriate compositions, we obtain the system $(K_x, k_{yx}\mid x,y \in \tilde J \;\&\; x\preceq y)$ which is easily seen to be direct, it has $M$ as its direct limit, and $(K_{(i,1)}\mid i\in J)$ forms a cofinal subsystem. Therefore, $M$ is a direct limit of this cofinal subsystem, which clearly consists of modules from $\mathcal A$.

\end{proof}

Now, we can prove the crucial statement regarding cogeneration of cotorsion pairs by a single pure-injective module. To this end, we need the following notion from \cite[Section 9.4]{P}: A pure-injective module $N$ is said to be an \emph{elementary cogenerator} if every pure-injective direct summand of a module elementarily equivalent to $N^{\aleph _0}$ is a direct summand of some power of $N$. Further recall that the \emph{dual module} $M^d$ of a module $M$ is defined as $M^d = \Hom_\Z(M,\Q/\Z)$. It is a well-known fact that any module $M$ is an elementary submodel in its double dual $M^{dd}$ as well as in any reduced $\mathfrak F$-power $M^I/\Sigma_{\mathfrak F} M^I$ provided that $\mathfrak F$ is an ultrafilter on $\mathfrak P(I)$ (cf.\ Definition~\ref{def:f_prod}, these reduced powers are called \emph{ultrapowers}).

\begin{prop} \label{prop:single}
Let $(\mathcal A, \mathcal B)$ be a complete cotorsion pair with $\mathcal B$ closed under direct limits. Then there exists a pure-injective module $E$ such that the class $\Ker \Ext ^1_R(-,E)$ coincides with the class of all pure-epimorphic images of modules from $\mathcal A$. Moreover, $E$ can be taken of the form $\prod_{k\in K}E_k$, with $E_k$ indecomposable for each $k\in K$.
\end{prop}

\begin{proof} First of all, since $\mathcal B$ is closed under direct products and direct limits, it is closed under ultrapowers as well. Thence $M\in\mathcal B$ implies by Frayne's Theorem that $N\in\mathcal B$ provided that $N$ is a pure-injective direct summand of a module elementarily equivalent to $M$. In particular, $\mathcal B$ is closed under taking double dual modules.

If we denote by $(\mathcal D, \mathcal E)$ the cotorsion pair cogenenerated by the class of all pure-injective modules from $\mathcal B$, then $\mathcal D$ is exactly the class of all pure-epimorphic images of modules from $\mathcal A$ (cf. \cite[Lemmas 2.1 and 2.2]{AT}; here, the completeness of $(\mathcal A, \mathcal B)$ and $\mathcal B$ being closed under \emph{double duals} are actually needed).

By \cite[Corollary 9.36]{P}, for every module $M$ there exists an elementary cogenerator elementarily equivalent to $M$. Thus, by the first paragraph, we may consider a representative set $\mathcal S$ consisting of elementary cogenerators in $\mathcal B$ \st any module in $\mathcal B$ is elementarily equivalent to a module from $\mathcal S$. Now define $E$ to be the direct product of all modules from $\mathcal S$. To finish the main part of our proof, it is enough to show that any pure-injective module from $\mathcal B$ is in $\Prod(E)$, the class of all direct summands of powers of $E$. This is sufficient since then the left-hand class of the cotorsion pair cogenerated by $\{E\}$ will coincide with $\mathcal D$.

Let, therefore, $M\in \mathcal B$ be a pure-injective module and $N \in \mathcal S$ be a module elementarily equivalent to $M$. By \cite[Proposition 2.30]{P}, $M$ is a pure submodule (hence a direct summand) in a module elementarily equivalent to $N^{\aleph _0}$. Thus $M$ is a direct summand of some power of $N$ by the definition of elementary cogenerator.

To prove the moreover statement, first recall that, by a well-known result of Fischer, $E = PE(\bigoplus _{j\in J} E_j)\oplus F$ where $PE$ stands for pure-injective hull, $E_j$ is indecomposable pure-injective for each $j\in J$, and $F$ has no indecomposable direct summands; it may happen that $J$ is empty or $F = 0$. By \cite[Corollary 4.38]{P}, $F$ is a direct summand of a direct product, say $\prod _{l\in L}E_l$, of indecomposable pure-injective direct summands of modules elementarily equivalent to $E$. According to the first paragraph, $E_l\in\mathcal B$ for every $l\in L$. It follows that $PE(\bigoplus _{j\in J} E_j)\oplus\prod _{l\in L} E_l$ cogenerates the same cotorsion pair as $E$ does. Further, $PE(\bigoplus _{j\in J} E_j)$ is a direct summand in $\prod _{j\in J} E_j$ and the latter module is in $\mathcal B$ since it is elementarily equivalent to $PE(\bigoplus _{j\in J} E_j)\in\mathcal B$. (Here, we use the fact that the direct sum is an elementary submodel in its pure-injective hull as well as in the direct product.) Thus, again, $\prod_{k\in J\cup L} E_k$ cogenerates the same cotorsion pair as $E$ did.
\end{proof}

We are in a position to state the main result of this section. It is in fact an immediate consequence of the previous statements.

\begin{thm} \label{thm:single_tcmc}
Let $\mathfrak C = (\mathcal A, \mathcal B)$ be a complete cotorsion pair with both classes closed under direct limits. Then $\mathfrak C$ is cogenerated by a direct product of indecomposable pure-injective modules.
\end{thm}

\begin{proof} This follows easily by Theorem~\ref{thm:lim->pure} and Proposition~\ref{prop:single}.
\end{proof}

\begin{rem}
(1). Note that if $R$ is an artin algebra or, more generally, a semi-primary ring and $(\mathcal A,\mathcal B)$ is a projective cotorsion pair satisfying the hypotheses of TCMC, it follows from~\cite[Corollary 4.5]{KS2} that the class $\mathcal B$ is also of the form $\Ker \Ext ^1_R(-, N)$ for a pure-injective module $N$.

(2). The distinction between closure under direct limits and closure under pure-epimorphic images is rather subtle. The two notions often coincide, but no example of a (hereditary) cotorsion pair $(\mathcal A, \mathcal B)$ with $\mathcal A$ closed under direct limits and \emph{not} closed under pure-epimorphic images is known to the authors as yet.
\end{rem}

% = Applications to triangulated categories =========================

\section{Compactly generated triangulated categories}
\label{sec:triang}

In this section, we compare the results we have obtained above with the work of Krause on smashing localizations of triangulated categories in~\cite{K,K2}. As mentioned before, there is a bijective correspondence between smashing localizing pairs in the stable module category and certain cotorsion pairs in the usual module category which works for self-injective artin algebras~\cite{KS}. However, as we want to indicate now, there are strong analogues of both settings well beyond where the correspondence from~\cite{KS} works. First, we will recall some necessary terminology.

Let $\T$ be a \emph{triangulated category} which admits arbitrary (set indexed) coproducts. We will not define this concept here since it is well-known and the definition is rather complicated, but we refer for example to~\cite[IV]{GM}, \cite{H} or~\cite[\S 3]{Ke}. We say that an object $C \in \T$ is \emph{compact} if the canonical map $\bigoplus_i \Hom_\T(C,X_i) \to \Hom_\T(C,\coprod_i X_i)$ is an isomorphism for any family $(X_i)_{i \in I}$ of objects of $\mathcal T$. Here, we will denote coproducts in $\T$ by the symbol $\coprod$ to distinguish them from direct sums of abelian groups. Let us denote by $\Tcomp$ the full subcategory of $\T$ formed by the compact objects. The category $\T$ is then called \emph{compactly generated} if
\begin{enumerate}
\item $\Tcomp$ is equivalent to a small category.
\item Whenever $X \in \T$ \st $\Hom_\T(C,X) = 0$ for all $C \in \Tcomp$, then $X = 0$.
\end{enumerate}

As an important example here, let $R$ be a \emph{quasi-Frobenius ring}, that is a ring for which projective and injective modules coincide, and let $\stModR$ be the \emph{stable category}, that is the quotient of $\ModR$ modulo the projective modules. Then $\stModR$ is triangulated~\cite{H} and compactly generated~\cite[\S 1.5]{K}. Moreover, compact objects are precisely those isomorphic in $\stModR$ to finitely generated $R$-modules. Other examples of compactly generated triangulated categories are unbounded derived categories of module categories and the stable homotopy category.

Let $\mathcal X$ be a full triangulated subcategory of $\mathcal T$. Then $\mathcal X$ is called \emph{localizing} if $\mathcal X$ is closed under forming coproducts with respect to $\mathcal T$. We call $\mathcal X$ \emph{strictly localizing} if the inclusion $\mathcal X \to \T$ has a right adjoint. Finally, $\mathcal X$ is said to be \emph{smashing} if the right adjoint preserves coproducts. Note that being a smashing subcategory is stronger than being strictly localizing, which in turn is stronger than being a localizing subcategory.

A localizing subcategory $\mathcal X \subseteq \T$ is \emph{generated} by a class $\mathcal C$ of objects in $\T$ if it is the smallest localizing subcategory of $\T$  containing $\mathcal C$. Notice that $\T$ itself is generated by $\Tcomp$ as a localizing subcategory (cf.~\cite[\S 5]{R} or \cite[Theorem 2.1]{N2}).

As in~\cite{KS}, we define $(\mathcal X,\mathcal Y)$ to be a \emph{localizing pair} if $\mathcal X$ is a strictly localizing subcategory of $\T$ and $\mathcal Y = \Ker \Hom_\T(\mathcal X,-)$. The objects in $\mathcal Y$ are then called \emph{$\mathcal X$-local}. Note that this definition makes sense also for non-compactly generated triangulated categories and with this in mind, $(\mathcal X,\mathcal Y)$ is a localizing pair in $\T$ \iff $(\mathcal Y,\mathcal X)$ is a localizing pair in $\T^{op}$. Moreover, the class $\mathcal X$ is smashing \iff the class $\mathcal Y$ of all $\mathcal X$-local objects is closed under coproducts.

There is a useful analogue of countable direct limits in a triangulated category, called a homotopy colimit. Let
$$ X_0 \overset{\varphi_0}\to X_1 \overset{\varphi_1}\to X_2 \overset{\varphi_2}\to \dotsb $$
be a sequence of maps in $\T$. A \emph{homotopy colimit} of the sequence, denoted by $\hocolim X_i$, is by definition an object $X$ which occurs in the triangle
$$ \coprod_{i<\omega} X_i \overset{\Phi}\to \coprod_{i<\omega} X_i \to X \to \coprod_{i<\omega} X_i[1] \eqno{(\ddag)} $$
where the $i$-th component of the map $\Phi$ is the composite
$$
X_i \overset{(^{\;\;\mathrm{id}}_{-\varphi_i})}\to X_i \sqcup X_{i+1} \overset{j}\to \coprod_{i<\omega} X_i
$$
and $j$ is the split monomorphism to the coproduct. Note that a homotopy colimit is unique up to a (non-unique) isomorphism.
As an easy but important fact, we point up that when applying the functor $\Hom_\T(-,Z)$ on $(\ddag)$ for any $Z \in \T$, we get an exact sequence
{\small
$$
0 \leftarrow {\varprojlim}^1 \Hom_\T(X_i,Z) \leftarrow \prod \Hom_\T(X_i,Z) \overset{\Phi^*}\leftarrow \prod \Hom_\T(X_i,Z) \leftarrow \varprojlim \Hom_\T(X_i,Z) \leftarrow 0
$$
}

\noindent
where $\Phi^* = \Hom_\T(\Phi,Z)$ and $\varprojlim^1$ is the first derived functor of inverse limit.
%see~\cite[\S 3.5]{W}.

% - End of definitions for triangulated categories ------------------

Having recalled the terminology, we also recall the crucial correspondence between cotorsion pairs and localizing pairs shown in \cite{KS}:

\begin{thm} \label{thm:triang_to_cot}
Let $R$ be a self-injective artin algebra, $\ModR$ the category of all right $R$-modules and $\stModR$ the stable category. Then the assignment
$$ (\mathcal A,\mathcal B) \to (\underline{\mathcal A},\underline{\mathcal B}) $$
gives a bijective correspondence between projective cotorsion pairs in $\ModR$ and localizing pairs in $\stModR$. Moreover, the following hold:
\begin{enumerate}
\item $\underline{\mathcal A}$ is smashing in $\stModR$ \iff both $\mathcal A$ and $\mathcal B$ are closed under direct limits in $\ModR$.
\item $\underline{\mathcal A}$ is generated, as a localizing subcategory in $\stModR$, by a set of compact objects \iff $(\mathcal A, \mathcal B)$ is a cotorsion pair of finite type in $\ModR$.
\end{enumerate}
\end{thm}

\begin{proof}
This is an immediate consequence of~\cite[Theorem 7.6 and Corollary 7.7]{KS} and~\cite[Corollary 4.6]{AST}.
\end{proof}

We have proved in Theorem~\ref{thm:main} that any cotorsion pair $(\mathcal A,\mathcal B)$ coming from a smashing localizing pair is of countable type. We show that it is possible to state a similar countable type result for $\stModR$ purely in the language of triangulated categories.

\begin{defn}
Let $\T$ be a compactly generated triangulated category. We call an object $X \in \T$ \emph{countable} if it is isomorphic to the homotopy colimit of a sequence of maps $X_0 \overset{\varphi_0}\to X_1 \overset{\varphi_1}\to X_2 \overset{\varphi_2}\to \dotsb$ between compact objects. Furthermore, let $\T_\omega$ stand for the full subcategory of $\T$ formed by all countable objects.
\end{defn}

Note that $\T_\omega$ is skeletally small. Now we can state the following theorem:

\begin{thm} \label{thm:main_triang}
Let $R$ be a self-injective artin algebra and $\mathcal T = \stModR$ the stable category of right $R$-modules. Then every smashing subcategory of $\T$ is generated, as a localizing subcategory of $\T$, by a set of countable objects.
\end{thm}

We postpone the proof until after a few preparatory observations and lemmas. First note that countable objects in $\stModR$ for a self-injective algebra $R$ are precisely those isomorphic in $\stModR$ to countably generated modules from $\ModR$, see~\cite[Lemma 4.3]{R}.

Next, we recall a technical statement concerning vanishing of derived functors of inverse limits. We recall that $\varprojlim^k$ stands for the $k$-th derived functor of inverse limit and, for convenience, we let $\aleph_{-1} = 1$.

\begin{lem} \label{lem:cdim} \cite{Mi}
Let $R$ be a ring and $I$ be a directed set whose smallest cofinal subset has cardinality $\aleph_\alpha$, where $\alpha$ is an ordinal number or $-1$. Put
$$ d = \sup \{k < \omega \mid {\varprojlim}^k N_i \ne 0 \textrm{ for some } (N_i)_{i \in I^{op}} \} $$
where $(N_i)_{i \in I^{op}}$ stands for an inverse system of right $R$-modules indexed by $I^{op}$. Then $d = \alpha+1$ if $\alpha$ is finite and $d = \omega$ if $\alpha$ is an infinite ordinal number.
\end{lem}

The latter lemma has important consequences for direct limits that are ``small enough''. Recall that given a class $\mathcal C$ of modules, we denote by $\Add\,\mathcal C$ the class of all direct summands of arbitrary direct sums of modules in $\mathcal C$.

\begin{lem} \label{lem:lim_pres}
Let $R$ be a ring and $(M_i)_{i \in I}$ be a direct system of $R$-modules \st $\vert I \vert < \aleph_\omega$. Then there is an exact sequence:
$$
0 \to X_n \to \dotsb \to X_1 \to X_0 \to \varinjlim M_i \to 0,
$$
where $n$ is a non-negative integer and $X_j \in \Add\,\{M_i \mid i\in I\}$ for all $j = 0, \dots, n$.
\end{lem}

\begin{proof}
Consider the canonical presentation of $\varinjlim M_i$:
$$
\dotsb \overset{\delta_2}\to
\bigoplus_{i_0 < i_1 < i_2} M_{i_0 i_1 i_2} \overset{\delta_1}\to
\bigoplus_{i_0 < i_1} M_{i_0 i_1} \overset{\delta_0}\to
\bigoplus_{i_0 \in I} M_{i_0} \to \varinjlim M_i \to 0,
$$
where $M_{i_0 i_1 \dots i_k} = M_{i_0}$ for all $k$-tuples $i_0 < i_1 < \dotsb < i_k$ of elements of $I$. This is an exact sequence and it follows from~\cite{J} that
$$ {\varprojlim}^k \Hom_R(M_i,Y) = \Ker \Hom_R(\delta_k,Y) / \Img \Hom_R(\delta_{k-1},Y) $$
for any $R$-module $Y$ and any $k \geq 0$ (we let $\delta_{-1} = 0$ here). If we take the smallest $n$ \st $\vert I \vert \leq \aleph_n$ and $Y = \Ker \delta_n$, it follows from Lemma~\ref{lem:cdim} that the inclusion
$$ 0 \to \Ker \delta_n \to \bigoplus_{i_0 < i_1 < \dotsb < i_{n+1}} M_{i_0 i_1 \dots i_{n+1}} $$
splits since $\varprojlim^{n+2} \Hom_R(M_i,Y) = 0$ in this case. The claim of the lemma follows immediately.
\end{proof}

\begin{cor} \label{cor:lim_cl}
Let $R$ be a quasi-Frobenius ring and let $\underline{\mathcal A}$ be a localizing subcategory of $\stModR$. Assume that $(M_i)_{i \in I}$ is a direct system in $\ModR$ \st $\vert I \vert < \aleph_\omega$ and $M_i$ is an object of $\underline{\mathcal A}$ for each $i \in I$. Then also $\varinjlim M_i$ is an object of $\underline{\mathcal A}$.
\end{cor}

\begin{proof}
Note that any localizing subcategory is closed under direct summands~\cite{BN}. Then the claim follows immediately from the preceding lemma when taking into account that triangles in $\stModR$ correspond to short exact sequences in $\ModR$ and that the canonical functor $\ModR \to \stModR$ preserves coproducts.
\end{proof}

Now we are in a position to prove the theorem.

\begin{proof}[Proof of Theorem~\ref{thm:main_triang}]
Let $\mathcal{\underline A}$ be a smashing subcategory of $\T = \stModR$ and let $(\mathcal A,\mathcal B)$ be the corresponding projective cotorsion pair in $\ModR$ with $\mathcal B$ closed under direct limits given by Theorem~\ref{thm:triang_to_cot}. Then by Theorem~\ref{thm:main}, there is a set $\mathcal S$ of countably generated $R$-modules that generates the cotorsion pair.
% We can without loss of generality assume that $\mathcal S$ is closed under extensions and taking countably generated syzygies and cosyzygies---this is since $\mathcal A$ has those closure properties.

Let us denote by $\mathcal L$ the localizing subcategory of $\T$ generated by $\mathcal S$, viewed as set of (countable) objects of $\T$. We claim that then for each $X \in \T$, there is a triangle $X \overset{w_X}\to B_X \to L_X \to X[1]$ in $\T$ \st $B_X \in \mathcal{\underline B}$ and $L_X \in \mathcal L$.

Let us assume for a moment that we have proved the claim and let $A \in \mathcal{\underline A}$. If we consider the shifted triangle $L_A[-1] \overset{f}\to A \overset{w_A}\to B_A \to L_A$, then clearly $w_A = 0$ and $f$ is split epi. Hence, $A$ is a direct summand of $L_A[-1]$ and consequently, since $\mathcal L$ is closed under direct summands by~\cite{BN}, $A \in \mathcal L$. Thus, $\mathcal{\underline A} = \mathcal L$ and the theorem follows.

Therefore, it remains to prove the claim. Let $X \in\mathcal T$. If we view $X$ as an $R$-module, we can construct a special $\mathcal B$-preenvelope $0 \to X \to B_X \to L_X \to 0$ following the lines of~\cite[Theorem 3.2.1]{GT}: We construct a well-ordered continuous chain
$$
B_0 \subseteq B_1 \subseteq B_2 \subseteq \dotsb \subseteq B_\alpha \subseteq \dotsb
$$
indexed by ordinal numbers \st $B_0 = X$ and $B_{\alpha+1}$ is a \emph{universal extension} of $B_\alpha$ by modules from $\mathcal S$. That is, there is an exact sequence of the form:
$$ 0 \to B_\alpha \to B_{\alpha+1} \to \bigoplus_{j \in J_\alpha} Y_j \to 0, $$
where $Y_j$ is isomorphic to a module from $\mathcal S$ for each $j \in J_\alpha$ and the connecting homomorphisms $\delta_Z: \Hom_R(Z,\bigoplus_{j \in J} Y_j) \to \Ext^1_R(Z,B_\alpha)$ are surjective for all $Z \in \mathcal S$. In particular, $\Ext^1_R(Z,-)$ applied on $B_\alpha \subseteq B_\beta$ for any $\alpha<\beta$ gives the zero map. Since all the modules in $\mathcal S$ are countably presented, any morphism $\Omega(Z) \to B_{\aleph_1}$ in $\ModR$, where $Z \in \mathcal S$, factors through the inclusion $B_\alpha \subseteq B_{\aleph_1}$ for some $\alpha<\aleph_1$. It follows that $\Ext^1_R(Z,B_{\aleph_1}) = 0$ for each $Z \in \mathcal S$; hence $B_{\aleph_1} \in \mathcal B$. Now, if we set $L_\alpha = B_\alpha/X$ for each $\alpha$, we have a well-ordered continuous chain
$$
L_0 \subseteq L_1 \subseteq L_2 \subseteq \dotsb \subseteq L_\alpha \subseteq \dotsb
$$
\st $L_{\alpha+1}/L_\alpha \cong B_{\alpha+1}/B_\alpha \in \Add\,\mathcal S$. It follows from Eklof's Lemma (\cite[Lemma 3.1.2]{GT} or \cite[Lemma 1]{ET}) that $L_\alpha \in \mathcal A$ for each ordinal $\alpha$. Hence, $0 \to X \to B_{\aleph_1} \to L_{\aleph_1} \to 0$ is a special $\mathcal B$-preenvelope of $X$.

Now let us focus on the corresponding triangle $X \to B_{\aleph_1} \to L_{\aleph_1} \to X[1]$ in $\T$. Clearly $B_{\aleph_1} \in \mathcal{\underline B}$. Moreover, it follows by a straightforward transfinite induction on $\alpha$ that $L_\alpha \in \mathcal L$ for each $\alpha \leq \aleph_1$. For $\alpha=0$, obviously $L_0 = 0 \in \mathcal L$. To pass from $\alpha$ to $\alpha+1$, we use the fact that the third term in the triangle $L_\alpha \to L_{\alpha+1} \to \coprod_{j \in J_\alpha} Y_j \to L_\alpha[1]$ is in $\Add\,\mathcal S$. Finally, limit steps are taken care of by Corollary~\ref{cor:lim_cl}. The claim is proved and so is the theorem.
\end{proof}

Inspired by Theorem~\ref{thm:main_triang}, we can ask the following question:

\medskip

\noindent
{\bf Question (Countable Telescope Conjecture).} Let $\T$ be an arbitrary compactly generated triangulated category. Is every smashing localizing subcategory of $\T$ generated by a set of countable objects?\footnote{An affirmative and far more general answer to this question was given by Krause in~\cite[\S 7.4]{K3} after submission of this paper.}

\medskip

In this context, it is a natural question if one can characterize the countable objects in a smashing subcategory of a triangulated category. That is, we are looking for a triangulated category analogue of Theorem~\ref{thm:char_A}. It turns out that there is an analogous statement that holds for any compactly generated triangulated category.

\begin{thm} \label{thm:char_A_triang}
Let $\T$ be a compactly generated triangulated category and let $\mathcal X$ be a smashing subcategory of $\T$. Denote by $\mathfrak I$ the ideal of all morphisms between compact objects which factor through some object in $\mathcal X$. Then the following are equivalent for a countable object $X \in \T$:
\begin{enumerate}
\item $X \in \mathcal X$,
\item $X$ is the homotopy colimit of a countable direct system $(X_n,\varphi_n)$ of compact objects \st $\varphi_n \in \mathfrak I$ for every $n$.
\end{enumerate}
\end{thm}

\begin{proof}
$(1) \implies (2)$. Since $X$ is countable, we have $X = \hocolim Y_n$ where $(Y_n, \psi_n)$ is a direct system of compact objects (not necessarily from $\mathcal X$). Let $Z$ be an $\mathcal X$-local object and let $\tilde Z = \coprod_{i<\omega} Z_i$, where $Z_i = Z$ for each $i<\omega$. By assumption, $\tilde Z$ is also $\mathcal X$-local. If we apply $\Hom_\T(-,\tilde Z)$ on the triangle $\coprod_n Y_n \overset{\Phi}\to \coprod_n Y_n \to X\to \coprod_n Y_n[1]$, we see that $\Hom_\T(\Phi,\tilde Z)$ is an isomorphism. Hence we get:
$$\varprojlim \Hom_\T(Y_n,\tilde Z) = 0 = {\varprojlim}^1 \Hom_\T(Y_n,\tilde Z) .$$
Note also that $\Hom_\T(Y_n,\tilde Z)$ is canonically isomorphic to $\Hom_\T(Y_n,Z)^{(\omega)}$ for each $n<\omega$ since all the $Y_n$ are compact. Consequently, the inverse system
$$ (\Hom_\T(Y_n,Z), \Hom_\T(\psi_n,Z))_{n<\omega} $$
is Mittag-Leffler by Proposition~\ref{prop:ML} and T-nilpotent by Lemma~\ref{lem:t-nil}. Since the class of all $\mathcal X$-local objects is closed under coproducts, we infer, as in the proof of Theorem~\ref{thm:char_A}, that there are some bounds for T-nilpotency common for all $\mathcal X$-local objects $Z$. In other words, there is a cofinal subsystem $(Y_{n_k}, \varphi_k \mid k<\omega)$ of the direct system $(Y_n, \psi_n)$ \st $\Hom_\T(\varphi_k,Z) = 0$ for all $k<\omega$ and $\mathcal X$-local objects $Z$. Note that $X \cong \hocolim_k Y_{n_k}$ since the homotopy colimit does not change when passing to a cofinal subsystem, \cite[Lemma 1.7.1]{N3}.

Finally, if $\varphi$ is a morphism in $\T$ \st $\Hom_\T(\varphi,Z) = 0$ whenever $Z$ is $\mathcal X$-local, then $\varphi$ factors through an object in $\mathcal X$ by~\cite[Lemmas 3.4 and 3.8]{K}. Hence, $\varphi_k \in \mathfrak I$ for each $k$ and we can just put $X_k = Y_{n_k}$.

$(2) \implies (1)$. If $X$ and $(X_n,\varphi_n)$ are as in the assumption, then, by Lemma~\ref{lem:t-nil},
$$
\varprojlim \Hom_\T(X_n,Z) = 0 = {\varprojlim}^1 \Hom_\T(X_n,Z)
$$
whenever $Z$ is $\mathcal X$-local. Thus, if we consider the triangle $\coprod_n X_n \overset{\Phi}\to \coprod_n X_n \to X\to \coprod_n X_n[1]$ defining $X$, then $\Hom_\T(\Phi,Z)$ is an isomorphism. For a similar reason, $\Hom_\T(\Phi[1],Z)$ is an isomorphism, and consequently $\Hom_\T(X,Z)=0$ for all $\mathcal X$-local objects $Z$. In other words: $X \in \mathcal X$.
\end{proof}

Triangulated category analogues of Theorems~\ref{thm:char_B} and~\ref{thm:single_tcmc}, the remaining main results of this paper, have been proved by Krause in~\cite{K}. We include the corresponding statements from~\cite{K} here to underline how straightforward the translation is. Let us start with Theorem~\ref{thm:char_B}---actually, \cite[Theorem A]{K} served as an inspiration for it:

\begin{thm} \label{thm:char_B_triang} \cite[Theorem A]{K}
Let $\T$ be a compactly generated triangulated category and let $\mathcal X$ be a smashing subcategory of $\T$. Denote by $\mathfrak I$ the ideal of all morphisms between compact objects which factor through some object in $\mathcal X$. Then the following are equivalent for $Y \in \T$:
\begin{enumerate}
\item $Y$ is $\mathcal X$-local,
\item $\Hom_\T(f,Y) = 0$ for each $f \in \mathfrak I$.
\end{enumerate}
\end{thm}

We conclude the paper with an analogue of Theorem~\ref{thm:single_tcmc}. Let us first recall that one defines pure-injective objects in a compactly generated triangulated category $\T$ as follows (see~\cite{K}): Let us call a morphism $X \to Y$ in $\T$ a \emph{pure monomorphism}
if the induced map $\Hom_\T(C,X) \to \Hom_\T(C,Y)$ is a monomorphism for every compact objects $C$. An object $X$ is then called \emph{pure-injective} if every pure monomorphism $X \to Y$ splits. As for module categories, the isomorphism classes of indecomposable pure-injective objects form a set which we call a spectrum of $\T$. The following has been proved in~\cite{K}:

\begin{thm} \label{thm:single_triang} \cite[Theorem C]{K}
Let $\T$ be a compactly generated triangulated category and let $\mathcal X$ be a smashing subcategory of $\T$. Then $X \in \mathcal X$ \iff $\Hom_\T(X,Y) = 0$ for each indecomposable pure-injective $\mathcal X$-local object $Y$.
\end{thm}

For stable module categories over self-injective artin algebras, the correspondence via Theorem~\ref{thm:triang_to_cot} works especially well because of the following result from \cite{K}:

\begin{prop} \label{prop:pi_triang} \cite[Proposition 1.16]{K}
Let $R$ be a quasi-Frobenius ring and $X$ be a right $R$-module. Then $X$ is a pure-injective module \iff $X$ is a pure-injective object in $\stModR$.
\end{prop}

\medskip

% = Bibliography ==============================================================


\begin{thebibliography}{AST}

\bibitem{AR} J.\ Ad\'amek and J.\ Rosick\'y, 
	\textbf{Locally Presentable and Accessible Categories}, 
	London Math.\ Soc.\ Lect.\ Note Ser., Vol.\ 189,
	Cambridge Univ. Press, Cambridge, 1994.

\bibitem{ABH} L.\ Angeleri H\" ugel, S.\ Bazzoni and D.\ Herbera,
	\emph{A solution to the Baer splitting problem},
	Trans.\ Amer.\ Math.\ Soc.\ \textbf{360} (2008), no.\ 5, 2409--2421.

\bibitem{AHK} L.\ Angeleri H\" ugel, D.\ Happel and H.\ Krause,
	\textbf{Handbook of Tilting Theory},
	London Math.\ Soc.\ Lect.\ Note Ser., Vol.\ 332,
	Cambridge Univ.\ Press, 2007.

\bibitem{AST} L.\ Angeleri H\" ugel, J.\ \v Saroch and J.\ Trlifaj,
	\emph{On the telescope conjecture for module categories},
	J.\ Pure Appl.\ Algebra \textbf{212} (2008), 297--310.

\bibitem{AT} L.\ Angeleri H\" ugel and J.\ Trlifaj,
	\emph{Direct limits of modules of finite projective dimension},
	Rings, Modules, Algebras, and Abelian Groups, LNPAM \textbf{236} (2004), 27--44.

\bibitem{ARS} M.\ Auslander, I.\ Reiten and S.\ O.\ Smal\o{},
	\textbf{Representation theory of Artin algebras},
	Cambridge Studies in Advanced Mathematics 36,
	Cambridge University Press, Cambridge, 1995.

\bibitem{BET} S.\ Bazzoni, P.\ C.\ Eklof and J.\ Trlifaj,
	\emph{Tilting cotorsion pairs},
	Bull.\ London Math.\ Soc.\ \textbf{37} (2005), no.\ 5, 683--696.

\bibitem{BH} S.\ Bazzoni and D.\ Herbera,
	\emph{One dimensional tilting modules are of finite type},
	Algebr.\ Represent.\ Theory \textbf{11} (2008), no.\ 1, 43--61.

\bibitem{BS} S.\ Bazzoni and J.\ \v S\v tov\'\i\v cek,
	\emph{All tilting modules are of finite type},
	Proc.\ Amer.\ Math.\ Soc.\ \textbf{135} (2007), no.\ 12, 3771--3781.

\bibitem{AB} A.\ Beligiannis
	\emph{Cohen-Macaulay modules, (co)torsion pairs and virtually Gorenstein algebras},
	Journal of Algebra \textbf{288} (2005), 137--211.

\bibitem{BN} M.\ B\"okstedt and A.\ Neeman,
	\emph{Homotopy limits in triangulated categories},
	Compositio Math.\ \textbf{86} (1993), no.\ 2, 209--234.

\bibitem{Bo} A.\ K.\ Bousfield,
	\emph{The localization of spectra with respect to homology},
	Topology \textbf{18} (1979), no.\ 4, 257--281.

\bibitem{C} W.\ W.\ Crawley-Boevey,
	\emph{Infinite dimensional modules in the representation theory of finite dimensional algebras},
	in Algebras and Modules~I, Canad.\ Math.\ Soc.\ Conf.\ Proc., Vol.\ 23, AMS, Providence 1998, 29--54.

\bibitem{E} P.\ C.\ Eklof,
	\emph{Shelah's Singular Compactness Theorem},
	Publ.\ Mat.\ \textbf{52} (2008), no.\ 1, 3--18.

\bibitem{EM} P.\ C.\ Eklof and A.\ H.\ Mekler,
	\textbf{Almost Free Modules}, 2nd Ed.,
	North-Holland Math. Library, Elsevier, Amsterdam 2002.

\bibitem{ET} P.\ C.\ Eklof and J.\ Trlifaj,
	\emph{How to make Ext vanish},
	Bull.\ London Math.\ Soc.\ \textbf{33} (2001), no. 1, 41--51.

\bibitem{FL} L.\ Fuchs and S.\ B.\ Lee,
	\emph{From a single chain to a large family of submodules},
	Port.\ Math.\ (N.S.) \textbf{61} (2004), no.\ 2, 193--205.

\bibitem{GM} S.\ I.\ Gelfand and Y.\ I.\ Manin,
	\textbf{Methods of homological algebra},
	Second edition. Springer Monographs in Mathematics.
	Springer-Verlag, Berlin, 2003.

\bibitem{GT} R.\ G\"obel and J.\ Trlifaj,
	\textbf{Approximations and endomorphism algebras of modules},
	de Gruyter Expositions in Mathematics \textbf{41},
	Berlin-New York 2006.

\bibitem{G} A. Grothendieck,
	\textbf{\'El\'ements de g\'eom\'etrie alg\'ebrique. III. \'Etude cohomologique des faisceaux coh\'erents},
	Inst.\ Hautes \'Etudes Sci.\ Publ.\ Math.\ No.\ 11, 1961.

\bibitem{H} D.\ Happel,
	\textbf{Triangulated categories in the representation theory of finite-dimensional algebras},
	London Mathematical Society Lecture Note Series, 119. Cambridge University Press, Cambridge, 1988.

\bibitem{Hi} P.\ Hill,
	\emph{The third axiom of countability for abelian groups},
	Proc.\ Amer.\ Math.\ Soc.\ \textbf{82} (1981), 347--350.

\bibitem{J} C.\ U.\ Jensen,
	\textbf{Les foncteurs d\'eriv\'e de $\varprojlim$ et
	leurs applications en th\'eorie des modules},
	Lectures Notes in Math. \textbf{254},
	Springer Verlag, Berlin-New York 1972.

\bibitem{JL} C.\ U.\ Jensen and H.\ Lenzing,
	\textbf{Model Theoretic Algebra},
	Gordon and Breach S.\ Publishers, 1989.

\bibitem{Ke} B.\ Keller,
	\emph{Introduction to abelian and derived categories},
	Representations of reductive groups, 41--61,
	Publ.\ Newton Inst., Cambridge Univ.\ Press, Cambridge, 1998. 

\bibitem{Ke2} B.\ Keller,
	\emph{A remark on the generalized smashing conjecture},
	Manuscripta Math.\ \textbf{84} (1994), no.\ 2, 193--198.

\bibitem{K2} H.\ Krause,
	\emph{Cohomological quotients and smashing localizations},
	Amer.\ J.\ Math.\ \textbf{127} (2005), no.\ 6, 1191--1246.

\bibitem{K3} H.\ Krause,
	\emph{Localization for triangulated categories},
	preprint.

\bibitem{K} H.\ Krause,
	\emph{Smashing subcategories and the telescope conjecture---an
	algebraic approach},
	Invent.\ Math.\ \textbf{139} (2000), no.\ 1, 99--133.

\bibitem{KS} H.\ Krause and \O.\ Solberg,
	\emph{Applications of cotorsion pairs},
	 J.\ London Math.\ Soc.\ (2) \textbf{68} (2003), no.\ 3, 631--650.

\bibitem{KS2} H.\ Krause and \O.\ Solberg,
	\emph{Filtering modules of finite projective dimension},
	 Forum Math.\ \textbf{15} (2003), 377--393.
	 
\bibitem{L} H.\ Lenzing,
	\emph{Homological transfer from finitely presented to infinite modules},
	Lectures Notes in Math. \textbf{1006},
	Springer, New York (1983), 734--761.

%\bibitem{M} F.\ Maeda,
%	\textbf{Kontinuerliche Geometrien},
%	Springer Verlag, Berlin, G\"ottingen, Heidelberg 1958.
%
\bibitem{Mi} B.\ Mitchell,
	\emph{The cohomological dimension of a directed set},
	Canad.\ J.\ Math.\ \textbf{25} (1973), 233--238.

\bibitem{N1} A.\ Neeman,
	\emph{The connection between the $K$-theory localization theorem of
	Thomason, Trobaugh and Yao and the smashing subcategories of Bousfield
	and Ravenel},
	Ann.\ Sci.\ \'Ecole Norm.\ Sup.\ (4) \textbf{25} (1992), no.\ 5, 547--566.

\bibitem{N2} A.\ Neeman,
	\emph{The Grothendieck duality theorem via Bousfield's techniques and Brown representability},
	J.\ Amer.\ Math.\ Soc.\ \textbf{9} (1996), no.\ 1, 205--236.

\bibitem{N3} A.\ Neeman,
	\textbf{Triangulated categories},
	Annals of Mathematics Studies, 148, Princeton University Press, Princeton, NJ, 2001.

\bibitem{P} M. Prest,
	\textbf{Model theory and modules},
	London Math. Soc. Lec. Note Ser. \textbf{130},
	Cambridge University Press, Cambridge, 1988.

\bibitem{Ra} D.\ Ravenel,
	\emph{Localization with respect to certain periodic homology theories},
	Amer.\ J.\ Math.\ \textbf{106} (1984), no.\ 2, 351--414.

\bibitem{R} J.\ Rickard,
	\emph{Idempotent modules in the stable category},
	J.\ London Math.\ Soc.\ (2) \textbf{56} (1997), no.\ 1, 149--170.

\bibitem{Sa} L.\ Salce,
	\emph{Cotorsion theories for abelian groups},
	Symposia Mathematica, Vol.\ XXIII (Conf.\ Abelian Groups and their Relationship to the Theory of Modules, INDAM, Rome, 1977),
	pp.\ 11--32, Academic Press, London-New York, 1979.

\bibitem{SaT} J.\ \v Saroch and J.\ Trlifaj,
	\emph{Completeness of cotorsion pairs},
	Forum Math.\ \textbf{19} (2007), 749--760.

\bibitem{ST2} J.\ \v S\v tov\'\i\v cek and J.\ Trlifaj,
	\emph{All tilting modules are of countable type},
	Bull. Lond. Math. Soc. \textbf{39} (2007), 121--132.

\bibitem{ST} J.\ \v S\v tov\'\i\v cek and J.\ Trlifaj,
	\emph{Generalized Hill lemma, Kaplansky theorem for cotorsion pairs,
	and some applications},
	to appear in Rocky Mountain J.\ Math.

\bibitem{W} C.\ A.\ Weibel,
	\textbf{An introduction to homological algebra},
	Cambridge Studies in Advanced Mathematics, \textbf{38}.
	Cambridge University Press, Cambridge, 1994.

\bibitem{Z} M.\ Ziegler,
	\emph{Model theory of modules},
	Ann.\ Pure Appl.\ Logic \textbf{26} (1984), 149--213.

\end{thebibliography}
\end{document}